\documentclass[a4paper,12pt]{article}
\usepackage{}
\usepackage{mathrsfs}
\usepackage{amsthm}
\usepackage{mathrsfs}
\usepackage{amsfonts}
\usepackage{amsmath}
\usepackage{amssymb}
\usepackage{amsthm}
\usepackage{enumerate}
\usepackage[mathscr]{eucal}
\usepackage{eqlist}
\usepackage{graphicx}
\usepackage{color}

\usepackage[body={16.5true cm, 24.5true cm}]{geometry}

\newtheorem{Theorem}{Theorem}[section]

\newtheorem{Lemma}[Theorem]{Lemma}

\newtheorem{Remark}[Theorem]{Remark}
\newtheorem{Proposition}[Theorem]{Proposition}

\numberwithin{equation}{section} \allowdisplaybreaks
\usepackage{mathdots}

\renewcommand\abstract{{\bf Abstract}}

\begin{document}
\title{Stability and rigidity results of space-like hypersurface in the Minkowski  space\footnote{\footnotesize This work is supported by the NSFC (Grant No.12001276, No.12201138) and Mathematics Tianyuan fund project (Grant No.12226350, No.12326303), and supported by Jiangxi Provincial Natural
Science Foundation (Nos.20242BAB23002, 20242BAB25001 and 20232BAB201001).}}

\author{Jianhua Chen$^{1}$, Haiyun Deng$^2$,  Haiqin Xie$^{3}$, Jiabin Yin$^{4}$\footnote{\footnotesize Corresponding author. E-mail: cjh19881129@163.com; hydeng@nau.edu.cn; 1540061932@qq.com; jiabinyin@126.com} \\[12pt]
\small \emph {$^{1}$School of Mathematics and Computer Science, Nanchang University, Nanchang, 330031, PR China;}\\
\small \emph {$^{2}$Department of Applied Mathematics, Nanjing Audit University, Nanjing, 211815, China}\\
\small \emph {$^{3}$School of Mathematics and Statistics, Guangxi Normal University, Guilin, 541004, PR China;}\\
\small \emph {$^{4}$School of Mathematics and Statistics, Xinyang Normal University, Xinyang, 464000, PR China;}\\}
\date{}
\maketitle

\renewcommand{\labelenumi}{[\arabic{enumi}]}

\begin{abstract}:
In this paper, we establish some rigidity theorems for space-like hypersurfaces in the Minkowski space by using a Weinberger-type approach with P-functions and integral identities.
Firstly, for space-like hypersurfaces $M$ represented as graphs  $x_{n+1}=u(x)$ over domain $\Omega\subset\mathbb R^n$, if higher-order mean curvature ratio $\frac{H_{k}}{H_l}(l<k)$ is constant and the boundary $\partial M$ lies on a hyperplane intersecting with constant angles, then the hypersurface must be a part of hyperboloid.
Secondly, for convex space-like hypersurfaces with boundaries on a hyperboloid or light cone, if higher-order mean curvature ratio $\frac{H_{k}}{H_l}(l<k)$ is constant and the angle function between the normal
vectors of the hypersurface and the  hyperboloid (or the lightcone)
on the boundary is constant, then such hypersurfaces must be a part of hyperboloid. These results significantly extend Gao's previous work presented in \cite{Gao1,Gao2}.

Furthermore, we derive two fundamental integral identities for constant mean curvature (CMC) graphical hypersurfaces $x_{n+1}=u(x)$, $x\in\Omega\subset\mathbb R^n$,  and  the boundary lies on a hyperplane.  As some applications:
we obtain some equivalence conditions for hyperboloid. We also present a stability estimate that states that a compact hypersurface
$\partial\Omega$ can be contained in a spherical annulus whose interior and exterior
radii, say $\rho_i$ and $\rho_e$, satisfy the inequality
$$
\rho_e-\rho_i\leq C\|H_{\partial\Omega}-H_0\|_{L^1(\partial\Omega)}^{\tau_n},
$$
where $\tau_n=1/2$ if $n=2,3$ and $\tau_n=1/(n+2)$ if $n\geq 4$. Here, $H_{\partial\Omega}$ is the mean curvature of $\partial\Omega$,  $H_0$ is some reference constant and $C$ is a constant that depends on some geometrical and dimensional.
 Finally, analogous estimates are established.
\end{abstract}



\section{Introduction and main results}~~
The characterization of hypersurfaces with constant curvature functions constitutes a classical problem in Differential Geometry. A foundational result in this direction is Alexandrov's celebrated \cite{AA}, which states that any  compact embedded hypersurface in the Euclidean space $\mathbb R^{n+1}$ with constant mean curvature $H$ must be a sphere. This result was extended to compact embedded hypersurfaces with constant higher-order mean curvature $H_r$  by  Ros \cite{Ros}, employing the Heintze-Karcher inequality. Subsequent generalizations by Koh and Lee \cite{Koh,KL} established that if a compact embedded hypersurface $\Sigma$ in $\mathbb R^{n+1}$, $H^{n+1}$ or in an open half sphere of $S^{n+1}$ satisfies $H_k/H_l=$constant for $0\leq k,l\leq n$ with non-vanishing $H_l$, then $\Sigma$  must be a hypershpere.

Parallel developments emerged from Serrin's seminal work on symmetry properties of solutions to overdetermined boundary value problems. Using Alexandrov's reflection principle \cite{AA}, Serrin proved that if a smooth bounded $\Omega\subset\mathbb R^n$($n\geq2$)  admits $u\in C^2(\bar\Omega)$ satisfying
\begin{equation}\label{eqn:1.1}
\left\{
\begin{aligned}
&\Delta u=n \ \ \  {\rm in}\ \Omega,\\
&u=0\ \ \ \ \ \  \ {\rm on }\ \partial\Omega,\\
&|Du|=R\ \ \ \ \  \ {\rm on }\ \partial\Omega,
\end{aligned}\right.
\end{equation}
then $u=\frac{|x|^2-R^2}2$ and $\Omega$ is a ball. Alternative approaches using integral identities, inequalities, and maximum principles were later developed by Weinberger \cite{We}. Some generalized results of overdetermined problem for Hessian equations and Hessian quotient equations  had proved in \cite{BNST,GJZ,GMY} and so on.

The two problems share several common features and relevance as follows \cite{Ma,MP}:
\begin{equation}\label{eqn:1.2}
\frac1{n-1}\int_{\Omega}\left\{|D^2u|^2-\frac{(\Delta u)^2}n\right\}\text{d}x+\frac1R\int_{\partial\Omega}(u_{\nu}-R)^2\text{d}S_x=\int_{\partial\Omega}(H-H_0)^2(u_{\nu})^2\text{d}S_x,
\end{equation}
and
\begin{equation}\label{eqn:1.3}
\frac1{n-1}\int_{\Omega}\left\{|D^2u|^2-\frac{(\Delta u)^2}n\right\}\text{d}x+\frac1R\int_{\partial\Omega}\frac{(1-H|Du|)^2}{H}\text{d}S_x=\int_{\partial\Omega}\frac1H\text{d}S_x-n|\Omega|.
\end{equation}
Here, $R=\frac{n|\Omega|}{|\partial\Omega|}$ and $H_0=\frac1R=\frac{|\partial\Omega|}{n|\Omega|}$, $\nu$ is unit normal vector of $\partial\Omega$ and $H$ is mean curvature of $\partial\Omega$.

Magnanini  etc. \cite{Ma,MP,MP3} establish some geometrical estimation of Alexandrov's Soap Bubble theorem and Serrin's problem based on   \eqref{eqn:1.2} and \eqref{eqn:1.3}. In 2019, Magnanini and Poggesi\cite{Ma} obtained a stability estimate that states that a compact hypersurface $\Gamma\subset\mathbb{R}^n$ can be contained in a spherical annulus whose interior and exterior radii, say $\rho_i$ and $\rho_e$, satisfy the inequality
\[\rho_e-\rho_i\leq C\|H-H_0\|^{\tau_n}_{L^1(\Gamma)},\]
where $\tau_n=1/2$ for $n = 2, 3$, and  $\tau_n=1/(N + 2)$ for $n\geq 4$. Here, $H$ is the mean curvature of $\Gamma$, $H_0$ is a constant that only depends on the geometrical and spectral parameters of $\Omega$. Besides, they improved the results by using the same method in \cite{MP}, and showed that
\[\rho_e-\rho_i\leq C\|H-H_0\|_{L^2(\Gamma)}, ~\mbox{for}~n=2~\mbox{or}~3,\]
and if $n\geq 4,$ then
\[\rho_e-\rho_i\leq C\|H-H_0\|^{2/(n+2)}_{L^2(\Gamma)}, ~\mbox{if}~\|H-H_0\|_{L^2(\Gamma)}<\epsilon,\]
for some positive constants $C$ and $\epsilon.$ These are outstanding work. For the related beautiful research work, we refer, for example, \cite{CM,CM2016,CM2018,MP3,MP2024,P2024,P2025}.

In this paper, we concern characterization of space-like hypersurfaces with constant mean curvature functions in Minkowski space. Let $\mathbb R^{n,1}$ be the Minkowski space with the Lorentzian metric
$$
\bar g=\sum_i^ndx_i^2-dx_{n+1}^2.
$$
The study of space-like hypersurfaces with constant curvatures in Minkowski space has generated considerable interest, particularly for noncompact cases. Notable contributions include Cheng and Yau's Bernstein-type theorem for maximal hypersurfaces \cite{CY}, Treiberg's construction of entire spacelike hypersurfaces with constant mean curvature \cite{T}, and existence results for constant higher-order mean curvatures by Li and Wang-Xiao \cite{Li,WX}.

It is shown in \cite{ALP,AM,AP} that hyperbolic caps and hyperplanar balls are the only spacelike hypersurfaces of constant mean curvature, provided their boundary is spherical. Unlike these cases, Gao \cite{Gao1,Gao2} consider an intersection angle boundary condition for the problem. Speak precisely,  Gao established rigidity theorems for space-like hypersurfaces with constant $k$-th mean curvature under specific boundary conditions. To formulate these results precisely, we recall key concepts:
Let $M$ be  a space-like hypersurface  of $\mathbb R^{n,1}$ with boundary
$\partial M$ on $\bar\Sigma$. Then, the  intersection angle of two hypersurfaces $M$ and $\bar\Sigma$ on $\partial M$, called {\it angle function $\theta$}, defined by
\begin{equation*}
\begin{aligned}
\theta(p)=\bar g(N(p),\bar N(p))\ \ \  {\rm for}\ p\in\partial M
\end{aligned}
\end{equation*}
where $N$ and $\bar N$ are the timelike unit normal vectors of $M$ and $\bar\Sigma$, respectively.
The $k$-th mean curvature $H_k$ (also called higher order mean curvature) of $M$ is
defined by
$$
H_k=S_{k}(\lambda),
$$
where $S_k$ is the $k$-th elementary symmetric polynomial of principal curvatures $(\lambda_1,\ldots,\lambda_n)$.

When  $\bar\Sigma=\mathbb R^n\times\{c\}$, the following hyperboloid characterization was obtained:
\begin{Theorem}[\cite{Gao1}]\label{thm:1}
Suppose $M$ is a connected, space-like hypersurface with boundary $\partial M$ in $\mathbb R^{n,1}$ which is a graph on the hyperplane $\mathbb R^n\times\{c\}$. If the $k$-th mean curvature of $M$ is constant, and $\partial M$ is on $\mathbb R^n\times\{c\}$ with constant intersection
angles, then $M$ must is either  a part of the hyperboloid or entirely contained in $\mathbb R^n\times\{c\}$.
\end{Theorem}

If $\bar\Sigma$ is a general (upper half) hyperboloid obtained from the upper half of the unit hyperboloid
$$\Sigma=\{x\in\mathbb R^{n,1}|x_{n+1}=\sqrt{1+x_1^2+\cdots+x_n^2}\},$$
by translation and dilation, another rigidity result was established:
\begin{Theorem}[\cite{Gao2}]\label{thm:2}
Suppose $M $ is a convex and connected space-like hypersurface in $\mathbb R^{n,1}$ with boundary $\partial M $ on a hyperboloid or a light cone $ \bar\Sigma$. If the angle function $\theta$ is constant and  $H_k$ is constant, then $M$ must be a part of the hyperboloid.
\end{Theorem}

Theorem \ref{thm:1} corresponds to an overdetermined boundary value problem for the $k$-curvature equation:
\begin{equation}\label{eqn:1.4}
S_k\left(D\left(\frac{Du}{\sqrt{1-|Du|^2}}\right)\right)=\binom{n}{k}\ \ {\rm  in }\ \Omega, \ \  u=c\ \ \ {\rm on}\ \partial\Omega,
\end{equation}
and
\begin{equation*}
|Du|=c_2\ \ \ {\rm on}\ \partial\Omega.
\end{equation*}
 Meanwhile Theorem \ref{thm:2} can be seem as $k$-curvature equation, too. Both results were proved by using adaptations of Weinberger-type's method involving $P$-functions and integral identities. In particular, Urbas \cite{Urbas} prove existence and uniqueness of Dirichlet problem \eqref{eqn:1.4} under the convexity condition.

In this paper, we extend these investigations to the case where the ratio
$\frac{H_k}{H_l}$ is constant for $l<k$, corresponding to Hessian quotient curvature equations. Mao etc. \cite{GGYM} prove the existence and uniqueness of Dirichlet problem of Hessian quotient curvature equation under convexity condition, other related works \cite{GJ,Wang}.
Our main results demonstrate that such curvature constraints, combined with appropriate boundary conditions, enforce rigid geometric structures on spacelike hypersurfaces in the Minkowski space.

\begin{Theorem}\label{main1}
Let $M$ be the graph of a smooth function $x_{n+1} = u(x)$, $x\in\Omega\subset\mathbb R^n\times\{c\}$, and suppose that $M$ is a connected space-like hypersurface with boundary $\partial M$ in $\mathbb R^{n,1}$. If $\partial M$ is on the hyperplane $\mathbb R^n\times\{c\}$ with constant intersection angles and
$\frac{H_k}{H_l}$ is constant for $l<k$, then $M$ must be a part of the hyperboloid.
\end{Theorem}
\begin{Theorem}\label{main2}
Let  $M $ be a convex and connected space-like hypersurface in $\mathbb R^{n,1}$ with boundary $\partial M $ on a hyperboloid or a light cone $ \bar\Sigma$. If the angle function $\theta$ is constant and  $\frac{H_k}{H_l}$ is constant for $l<k$, then $M$ must be a part of the hyperboloid.
\end{Theorem}

 Now, let $M$ be the graph of a smooth function $x_{n+1} = u(x)$, and suppose that $M$ is a space-like hypersurface with boundary $\partial M$ in $\mathbb R^{n,1}$. If $\partial M$ is on the hyperplane $\mathbb R^n\times\{c\}$  and mean curvature of $\partial\Omega$ is constant, then it corresponds to Dirichlet problem of mean curvature equation:
\begin{equation}
{\rm div}\left(\frac{Du}{\sqrt{1-|Du|^2}}\right)=n\ \ {\rm  in }\ \Omega, \ \  u=c\ \ \ {\rm on}\ \partial\Omega.
\end{equation}

Inspired by Magnanini's works in \cite{Ma,MP},  we obtain two integral identities \eqref{eqn:5.5} and \eqref{eqn:5.6} in the Theorem \ref{thm:5.1} and Theorem \ref{thm:5.2} respectively, which are similar with \eqref{eqn:1.2} and \eqref{eqn:1.3}.

As a consequence, we obtain the following theorem.
\begin{Theorem}\label{thm:1.5}
Let $M$ be the graph of a smooth function $x_{n+1} = u(x)$, $x\in\Omega\subset\mathbb R^n\times\{c\}$, and suppose that $M$ is a space-like hypersurface with boundary $\partial M$ in $\mathbb R^{n,1}$. If $\partial M$ is on the hyperplane $\mathbb R^n\times\{c\}$  and mean curvature of $M$ is constant, then the following statements are equivalent:\\
(1) $M$  is a part of the hyperboloid,\\
(2) $\frac1{H_{\partial\Omega}}=\frac{|Du|}{\sqrt{1-|Du|^2}}$,\\
(3) $H_{\partial\Omega}=H_0=\frac{|\partial\Omega|}{n|\Omega|}$,\\
(4) $\theta=-\sqrt{1+R^2_0}$, that is $|Du|=\sqrt{\frac{R_0^2}{1+R_0^2}}$,\\
(5) $\Omega$ is a ball,\\
where $R_0=\frac{n|\Omega|}{|\partial\Omega|}$, $\theta$ is angle function and $H_{\partial\Omega}$ is mean curvature of $\partial\Omega$.
\end{Theorem}

While the rigidity theorems for constant curvature ratios extend previous works in a natural direction, the most profound contribution of this paper lies in its pioneering stability analysis for spacelike hypersurfaces governed by the mean curvature equation in Minkowski space.

Prior stability results, such as those by Magnanini-Poggesi for the Alexandrov Soap Bubble Theorem and Serrin's problem, are fundamentally rooted in the linear Laplace equation. Translating this theory to the nonlinear mean curvature equation presents a significant challenge, as the governing PDE becomes quasilinear and the geometry more intricate.

Inspired by Magnanini-Poggesi' works, our work successfully bridges this gap. We establish two integral identities (Theorems 5.1 and 5.2) tailored to the mean curvature equation. From these identities, we derive quantitative stability estimates showing that a compact hypersurface with ''almost constant" mean curvature must be close to a ball which implies that $M$ is close to hyperboloid. Specifically, we prove that the boundary can be trapped in a spherical annulus whose width is controlled by the deviation of its mean curvature from a constant, measured in the 
$L^1$ or $L^2$ norm (see Theorem \ref{thm:5.8}).
 The proof is non-trivial, requiring a careful adaptation of the $P$-function method, refined elliptic estimates, Poincar\'e inequality and new way to relate $\Delta P$ to $\rho_e-\rho_i$.

\section{Preliminaries}
Throughout the paper we adopt the Einstein summation convention for repeated indices and use the following range convention of indices: $1\leq i,j,k,l,\ldots \leq n$ and $1\leq A,B,\ldots\leq (n+1)$.

\subsection{Elementary symmetric functions}~~~~
We recall some properties of elementary symmetric polynomials which will be used later.

 For $k\in \{1,\ldots, n\}$, the $k$-th elementary symmetric function of $\lambda=(\lambda_1,\ldots,\lambda_n)$ is defined by
$$
S_k(\lambda):=\sum_{1\leq i_{1}<\cdots<i_k\leq n}\lambda_{i_1}\cdots\lambda_{i_k}.
$$
Given  a real  $n\times n$ matrix $A=(h_{j}^i)$ with eigenvalues $\lambda(A)$, we can define the $k$-th elementary symmetric polynomial by $S_k(A):=S_k(\lambda(A))$. Thus
$$
S_k(\lambda)=S_k(A):=\frac1{k!}\delta_{j_1,\cdots, j_k}^{i_1,\cdots, i_k}h^{i_1}_{j_1}\cdots h^{i_k}_{j_k},
$$
where $\delta_{j_1,\cdots, j_k}^{i_1,\cdots,i_k}$ is the generalized Kronecker symbol defined by
\begin{equation*}
\delta_{j_1,\cdots, j_k}^{i_1,\cdots,i_k}:=
\left\{
\begin{aligned}
&1,\ \ \ \ \ \ \ \ \ {\rm if}\ (i_1,\cdots,i_k)\ \ {\rm is \ an \ even \ permutation\ of} \  (i_1,\cdots,i_k), \\
&-1,\ \  \ \ \ {\rm if}\ (j_1,\cdots,j_k)\ \ {\rm is \ an \ odd \ permutation\ of} \  (j_1,\cdots,j_k), \\
&0,\ \ \ \ \ \ \ \ \ {\rm otherwise}.
\end{aligned} \right.
\end{equation*}
We use the convention that $S_0= 1$ and $S_k=0$ for $k>n$.
Denote
$$
(S_{k})^{j}_i(A):=\frac{\partial S_k(A)}{\partial h^{i}_{j}}=\frac1{(k-1)!}\delta_{j_1,\cdots, j_{k-1},j}^{i_1,\cdots,i_{k-1},i}h^{i_1}_{j_1}\cdots h^{i_{k-1}}_{j_{k-1}}.
$$

The following properties of $S_k(A)$ can be found in \cite{Reilly}, see also \cite{DGX1} for non-symmetric matrices.
\begin{Proposition}\label{prop:2.1}
For any $n\times n$ matrix $A$, we have
\begin{equation}\label{eqn:2.01}
\left\{
\begin{aligned}
(S_k)^{j}_i(A)=&S_{k-1}(A)\delta_{j}^i-\sum_{l=1}^n(S_{k-1})^{l}_i(A)a_{l}^j,\\
(S_k)^{j}_i(A)(A)h_{j}^{k}h_{k}^{i}=&S_1(A)S_k(A)-(k+1)S_{k+1}(A),\\
(S_k)^{j}_i(A)\delta_{j}^{i}=&(n-k+1)S_{k-1}(A).
\end{aligned}\right.
\end{equation}
\end{Proposition}

Now we collect some inequalities related to the elementary symmetric functions.
The Garding cone $\Gamma^+_{k}$ is defined as
$$
\Gamma^+_{k}=\{\lambda\in\mathbb R^n|\, S_i>0,\ {\rm for}\ 1\leq i\leq k\}.
$$
We say $A$ belongs to $\Gamma^+_k$ if its eigenvalue $\lambda(A)\in \Gamma^+_k$.

\begin{Proposition}\,(Theorem 51 and 52 in \cite{HLP})
For $A=(a_{ij})\in\Gamma_k$ and $0\leq l< k\leq n$, $0\leq s< r\leq n$, $r\leq k$, $s\leq l$, we have
\begin{equation}\label{eqn:2.05}
\left(\frac{S_k(A)/\binom{n}{k}}{S_l(A)/\binom{n}{l}}\right)^{\frac1{k-l}}\leq
\left(\frac{S_r(A)/\binom{n}{r}}{S_s(A)/\binom{n}{s}}\right)^{\frac1{r-s}},
\end{equation}
and  the equality holds if and only if $\lambda_1=\lambda_2=\cdots=\lambda_n$.
\end{Proposition}

\begin{Proposition}\label{prop:2.3}(Theorem 15.18, \cite{Lie})
For $A=(a_{ij})\in\Gamma_k$ and $0\leq l<k\leq n$, the matrix $\left(\frac{\partial}{\partial a_{ij}}(\frac{S_k(A)}{S_l(A)})\right)$ is positive definite.
\end{Proposition}

\begin{Lemma}\label{lem:2.4}(Lemma 2.5 in \cite{GMY})
If
$$
\frac{S_k(A)}{S_l(A)}=\frac{\binom{n}{k}}{\binom{n}{l}}
$$
and $A\in\Gamma_k$ for $0\leq l<k\leq n$, then
\begin{equation*}
\frac{S_{k-1}(A)}{S_k(A)}\geq\frac{k}{n-k+1},\ \ \ \frac{S_{l+1}(A)}{S_l(A)}\geq \frac{n-l}{l+1}.
\end{equation*}
Each equality occurs if and only if $A=cI$ for some $c>0$,where $I$ is the identity matrix. Moreover, we also
have
\begin{equation*}
\frac{S_{k+1}(A)}{S_k(A)}\leq\frac{n-k}{k+1},\ \ \ \frac{S_{l-1}(A)}{S_l(A)}\leq \frac{l}{n-l+1}.
\end{equation*}

\end{Lemma}

\subsection{Hypersurface in Minkowski space $\mathbb R^{n,1}$}~~~~
We recall some facts and properties  of hypersurfaces in the Minkowski space and we refer to \cite{CY,Li}.

For convenience, write $\langle\cdot,\cdot\rangle=\bar g(\cdot,\cdot)$ for vectors in the Minkowski space $\mathbb R^{n,1}$, i.e.
$$
\langle x,y\rangle=x_1y_1+\cdots+x_ny_n-x_{n+1}y_{n+1}
$$
for $x=(x_1,\ldots,x_n,x_{n+1})$, $y=(y_1,\ldots,y_n)$.

A vector $x\in\mathbb R^{n,1}$ is called space-like, time-like or light-like if $\langle x,x\rangle>0$, $\langle x,x\rangle<0$ or $\langle x,x\rangle=0$, respectively.
Let $\mathbb R^{n,1}_+$ denotes the upper half of the Minkowski space, i.e.,
$$
\mathbb R^{n,1}_+:=\{x\in \mathbb R^{n,1}|\, x_{n+1}>0\}.
$$
Then Gao\cite{Gao2} proved the following fact.
\begin{Proposition}\label{prop:1}
Given $x,y\in\mathbb R^{n,1}_+$. If $x$ is time-like and y is time-like or light-like, then $\langle x,y\rangle<0$.
\end{Proposition}
\subsubsection{The graphic  hypersurface}~~~~
Suppose that smooth hypersurface $M$ can be  written as a graph $x_{n+1}=u(x_1,\ldots,x_n)$ on a domain $\Omega\subset\mathbb R^n$, the position vector of $M$ is $\phi=(x,u(x))$. If $M$ is space-like, then $$|Du|<1.$$
Let
$$
E_{1}=(1,0,\ldots,0),\ \ E_{2}=(0,1,\ldots,0)\ \ E_{n+1}=(0,0,\ldots,0,1),
$$
be a basis of $\mathbb R^{n,1}$, where $E_{n+1}$ is the time-like unit vector field.
Then
$$
\phi=x_AE_A.
$$
Let $\phi_i:=\frac{\partial\phi}{\partial x_i}=E_i+u_iE_{n+1}$.
We choose the time-like unit normal vector field of $M$,
$$
e_{n+1}=\frac{1}{w}(u_iE_i+E_{n+1}),
$$
where $w=\sqrt{1-|Du|^2}$.

Thus, we can obtain the induced metric and the second fundamental form of $M$, respectively. Moreover,
\begin{equation*}
ds^2=\langle d\phi,d\phi\rangle=(dx_i)^2-(u_idx_i)^2=g_{ij}dx_idx_j,
\end{equation*}
where
\begin{equation}\label{eqn:2.3}
g_{ij}=\delta_{ij}-\sum u_iu_j, \ \ \ g^{ij}=\delta_{ij}+\frac{u_iu_j}{1-|Du|^2},
\end{equation}
and
\begin{equation*}
\begin{aligned}
h=&\langle d\phi,de_{n+1}\rangle\\
=&\left\langle dx_iE_i+duE_{n+1},d\left(\frac{u_i}{w}\right)E_i+d\left(\frac1w\right)E_{n+1}\right\rangle\\
=&\frac1wdu_idx_i=\frac1w\frac{\partial^2 u}{\partial x_i\partial x_j} dx_idx_j,
\end{aligned}
\end{equation*}
that is
\begin{equation}\label{eqn:2.4}
h_{ij}=\frac{u_{ij}}w,
\end{equation}
where $w=\sqrt{1-|Du|^2}$.
It follows that
\begin{equation}\label{eqn:2.5}
h^{j}_i=g^{jk}h_{ik}=\frac{u_{ij}}{w}+\frac{u_ju_ku_{ik}}{w^3}=D_i\left(\frac{u_j}w\right).
\end{equation}
Let $\nabla$ denote the Levi-Civita connection on $(M,g)$.
Then we have
\begin{equation}\label{eqn:2.6}
\phi_{ij}=h_{ij}e_{n+1},\ \ \  e_{(n+1)i}=h_{i}^j\phi_j,
\end{equation}
where $\phi_{ij}:=\nabla_i\nabla_j\phi$ and $e_{(n+1)i}:=\nabla_ie_{n+1}$.

It is clear that
\begin{equation}\label{eqn:2.7}
D_mh_{i}^j=D_ih_{m}^j.
\end{equation}

The Codazzi equation of $M$ is $h_{ij,m}=h_{mj,i}$, which  implies
\begin{equation}\label{eqn:2.8}
h_{i,m}^j=h_{m,i}^j,\ \ \  {\rm and }\ \ \ \nabla^mh^{i}_j=\nabla^ih^m_j,
\end{equation}
where
$
h_{ij,m}=\nabla_mh_{ij} $ and $ h_{i,m}^j=\nabla_mh_{i}^j.
$
Then, by a direct computation (refer to Proposition 2.1 in \cite{Reilly} or Proposition 2.3 in \cite{DGX1}), we have the following  result.
\begin{Proposition}
If $A$ denotes the matrix $(h^{j}_i)$ defined in \eqref{eqn:2.5}, then $D_i(S_k^{ij}(A))=0$.
\end{Proposition}

\subsubsection {Convex hypersurface}

{\bf Gauss map:} Let $\phi:M\rightarrow\mathbb R^{n,1}$ be a space-like hypersurface. In local coordinates, $\{\partial_1,\ldots,\partial_n\} $ is a basis of tangent space $T_pM$ at $p\in M$. The induced metric of $M$ is
$$
g_{ij}=\langle\partial_i\phi,\partial_j\phi\rangle=\langle \phi_i,\phi_j\rangle.
$$
We can see that $g_{ij}$ is Riemannian if $M$ is space-like. This means all tangent vectors are space-like and the normal vector is time-like.

 For a connected, space-like hypersurface $M$, we choose a upward  time-like unit normal
vector of $M$, which means $e_{n+1}\in \mathbb R^{n,1}_{+}$.
As a result, the
mapping $p\mapsto e_{n+1}(p)$ take its values in the upper half of the hyperboloid
$$
\Sigma=\{x\in\mathbb R^{n,1}|\, \langle x,x\rangle=-1, x_{n+1}>0\},
$$
which is canonically embedded in $\mathbb R^{n,1}$ as hypersurface. Similar to the case in the Euclidean space, we can define Gauss map of the space-like hypersurface
$$
\mathcal{G}:M\rightarrow \Sigma;\ p\mapsto e_{n+1}(p).
$$

In local coordinates, the matrix $A=(h^j_i)$ of the differential $de_{n+1}$ is given by
$$
de_{n+1}(\frac{\partial}{\partial x_i})=e_{(n+1)i}=h^{j}_i\phi_{j},\ \ \ {\rm for}\  i=1,\ldots,n.
$$
We can see that $de_{n+1}$ is nondegenerate if $M$ is convex.

Thus the Gauss map $e_{n+1}:M\rightarrow e_{n+1}(M)$ is a diffeomorphism from a convex, space-like hypersurface $M$ to the image $e_{n+1}(M)$. Consequently, $M$ can be parametrized by $\Sigma$ in the sense of
$$
\phi(p)=\phi\circ e_{n+1}^{-1}(z)=:\phi(z)
$$
for $p\in M$ and $z=e_{n+1}(p)\in \Sigma$.

The Gauss formula  and Codazzi equation gives us, respectively
\begin{equation}\label{eqn:Gauss}
\nabla_i\nabla_j\phi=\phi_{ij}=h_{ij}e_{n+1}
\end{equation}
and
\begin{equation*}
h_{ij,s}=h_{is,j}.
\end{equation*}
Moreover, we have
\begin{equation*}
\nabla^jh^{s}_i=\nabla^sh^{j}_i.
\end{equation*}

\vskip 2mm
{\bf Inverse of $de_{n+1}$:} Define a function $u$ on $\phi(M)$ by
$$
u(z):=\langle \phi(z),z\rangle, \ \ \ {\rm for}\ z\in\phi(M).
$$
The position vectore $\phi$ of $M$ can be rewritten by $u$ as

\begin{equation}\label{eqn:2.9}
\phi(z)=Du(z)-u(z)z,
\end{equation}
where $Du$ is gradient of $u$. In local coordinates of $\Sigma$, we have
$$
Du=\tilde g^{sj}u_j\partial_sz=u^s\partial_sz,
$$
where $u_j=\frac{\partial u}{\partial x_j}$ and $\tilde g^{ij}$ is the inverse matrix of the metric $\tilde g_{ij}$
of $\Sigma$.

Let $\tilde g_{ij}$ and $\tilde h_{ij}$ denote the metric and the second fundamental form of the hyperboloid $\Sigma$, respectively. Since all principal curvatures of $\Sigma$ are equal to $1$, we can see that
$\tilde h_{ij}=\tilde g_{ij}$. Then
$$
\partial_i\partial_jz=\tilde\Gamma^s_{ij}\partial_sz+\tilde g_{ij}z,
$$
where $z$ is the is the position vector and $\tilde\Gamma^s_{ij}$ is the Christoffel symbol of $\Sigma$.

The differentiation of \eqref{eqn:2.9} give us
\begin{equation}\label{eqn:2.10}
\begin{aligned}
\partial_i\phi=&\partial_iu^s\partial_sz+u^s\partial_i\partial_sz-u_iz-u\partial_iz\\
=&\partial_iu^s\partial_sz+u^s\tilde\Gamma^m_{is}\partial_mz-u\partial_iz\\
=&(D_iu^s-u\delta^s_i)\partial_sz.
\end{aligned}
\end{equation}

Denote $b_i^j:=D_iu^j-u\delta^j_i$. Let $z=e_{n+1}$ in \eqref{eqn:2.10}, we know that the matrix $B=(b^j_i)$ is the inverse of $A=(h^{j}_i)$. Hence the eigenvalues of $B$ are reciprocals of the principal curvatures of $M$.

Thus, we have
\begin{equation}\label{eqn:2.12}
S_k(A)=\frac{S_{n-k}(B)}{S_n(B)}.
\end{equation}

Finally, Gao \cite{Gao2} proved the following results.
\begin{Proposition}\label{prop:2.6}(\cite{Gao2})
$b^j_{im}=b^j_{mi}$ and $D_i(S_k^{ij}(B))=0$.
\end{Proposition}

 \section{The proof of Theorem \ref{main1}}~~~~
In this section, we establish the rigidity result stated in Theorem \ref{main1}.

Under the hypotheses of Theorem~\ref{main1}, the problem reduces to an overdetermined boundary value problem for a Hessian-quotient curvature equation. To formalize this, we normalize the constant ratio of curvature functions as $$\frac{H_k}{H_l}=\frac{\binom{n}{k}}{\binom{n}{l}},$$
where $H_k$ is $k$-th mean curvature of $M$.
Then we have
\begin{equation}\label{eqn:3.1}
\frac{S_k(A)}{S_l(A)}=\frac{\binom{n}{k}}{\binom{n}{l}}\ \ \ \ {\rm in}\ \Omega,
\end{equation}
due to $\langle e_{n+1},E_{n+1}\rangle=-\frac{1}{\sqrt{1-|Du|^2}}=\theta_0$,
\begin{equation}\label{eqn:3.2}
u=c\ \ \  {\rm and}\ |Du|=c_2 \ \ \ {\rm on}\ \partial\Omega,
\end{equation}
where $c$ and $0<c_2=\sqrt{1-\theta_0^{-2}}<1$ are constants.

At first, we prove that $M$ is $k$-convex, i.e., the shape operator $A\in\Gamma_k$ everywhere in $M$. We follow
Jia's idea \cite{Jia} to prove the following Lemma.  This property is used to ensure that we can use the maximum principle and the Newton-MacLaurin inequality later. For the sake of completeness we provide a proof.
\begin{Lemma}\label{lem:3.1}
For $2\leq k\leq n$ and $0\leq l <k$,  if $u$ is a solution of \eqref{eqn:3.1} and \eqref{eqn:3.2}, then $A\in\Gamma_k$  in $\bar\Omega$.
\end{Lemma}

\begin{proof}
The proof for the case $l=0$ can found in \cite{Gao1}. Hence we assume $1\leq l<k\leq n$. From the boundless and the smoothness of
$\partial\Omega$, there exists at least one point $x_0\in\partial\Omega$, such that  the principal curvatures $\kappa_1,\ldots,\kappa_{n-1}$ of $\partial\Omega$ are nonnegative. The boundary condition implies $\partial\Omega$ is a level set of $u$, and also a hypersurface in $\mathbb R^{n}$. The outward unit normal vector of $\partial\Omega$ is $\nu=\frac{Du}{|Du|}$. Then we can choose a suitable coordinate at $x_0$ such that $e_n=\nu$ and $e_1,\ldots,e_{n-1}$ are tangent vectors on the principal directions of $\partial\Omega$ at $x_0$. Then we have

\begin{equation*}
D^2u=
\begin{pmatrix}
 c_2\kappa_1 &  & 0& u_{1n} \\
  & \ddots & & \vdots \\
 0& & c_2\kappa_{n-1}& u_{(n-1)n} \\
 u_{n1}&\cdots&u_{n(n-1)}& u_{nn}
\end{pmatrix}.
\end{equation*}
By the boundary conditions, we can see that $w=\sqrt{1-c_2^2}$, $u_{\alpha}=0$ and $u_{n\alpha}=0$  for $\alpha=1,\ldots,n-1$.
Then, we obtain from \eqref{eqn:2.4} that
\begin{equation*}
(h^{j}_i)=
\begin{pmatrix}
 c_2\frac{\kappa_1}{w} &  & &  \\
  & \ddots & &  \\
 & & c_2\frac{\kappa_{n-1}}{w}& \\
 & & &h^n_n
\end{pmatrix}.
\end{equation*}

Hence, from $\frac{S_k(A)}{S_l(A)}=\frac{\binom{n}{k}}{\binom{n}{l}}>0$, we have
\begin{equation}\label{eqn:3.3}
0<S_k(h^j_i)=h^n_nS_{k-1}(\tilde \kappa)+S_k(\tilde\kappa),
\end{equation}
where $\tilde\kappa=c_2\frac{\kappa}{\sqrt{1-c_2^2}}$.
This combining with Newton-MacLaurin inequality \eqref{eqn:2.5} yield to
$$
S_{k-1}(A)=h^n_nS_{k-2}(\tilde\kappa)+S_{k-1}(\tilde\kappa)\geq
-\frac{S_{k}(\tilde\kappa)S_{k-2}(\tilde\kappa)}{S_{k-1}(\tilde\kappa)}+S_{k-1}(\tilde\kappa)>0.
$$

Similarly, we obtain $S_m(A)>0$ for $m\in\{1,\ldots,k\}$ at $x_0$. Therefore $A\in\Gamma_k^+$ at $x_0$.
The above argument shows the set $\{x\in\bar\Omega|A(x)\in\Gamma^+_k\}$ is nonempty. Then the smoothness of $u$,
Newton-MacLaurin inequality and $S_k(A)>0$ imply $A\in \Gamma_k^+$ in $\bar\Omega$ (refer to \cite{GMY} for details).
\end{proof}

Next, we consider an auxiliary function
$$
P=\langle \phi,E_{n+1}\rangle-\langle e_{n+1},E_{n+1}\rangle.
$$

Let $F(A)=\frac{S_k(A)}{S_l(A)}$, $F^{j}_i=\frac{\partial F(A)}{\partial h^{i}_j}=F\left(\frac{(S_k)^j_i}{S_k}-\frac{(S_l)^j_i}{S_l}\right)$. From Proposition \ref{prop:2.3} and Lemma
\ref{lem:3.1}, we can see that $F^{j}_i\nabla_j\nabla^i=F^{j}_ig^{il}\nabla_j\nabla_l$ is  elliptic.

\begin{Lemma}\label{lem:3.2}
Let $u$ be a solution of \eqref{eqn:3.1}, then
$$
F^{j}_iP_j^i\geq0.
$$
Moreover, either
\begin{equation*}
P=-c+\frac{1}{\sqrt{1-c_2^2}} \ \ \  {\rm in}\ \bar\Omega
\end{equation*}
or
$$
P<-c+\frac{1}{\sqrt{1-c_2^2}} \ \ \  {\rm in}\ \Omega.
$$
\end{Lemma}

\begin{proof}
 By direct computation and \eqref{eqn:2.6}, we have
\begin{equation*}
\begin{aligned}
P_i=&\langle \phi_i,E_{n+1}\rangle-\langle e_{(n+1)i},E_{n+1}\rangle\\
=&\langle \phi_i,E_{n+1}\rangle-h^{s}_i\langle \phi_s,E_{n+1}\rangle.
\end{aligned}
\end{equation*}
Moreover, we have
\begin{equation*}
\begin{aligned}
P_{ij}=&\langle \phi_{ij},E_{n+1}\rangle-h^{s}_{i,j}\langle \phi_s,E_{n+1}\rangle-h^{s}_i\langle \phi_{sj},E_{n+1}\rangle\\
=&(h_{ij}-h^{s}_ih_{sj})\langle e_{n+1},E_{n+1}\rangle+h^{s}_{i,j}u_s.
\end{aligned}
\end{equation*}

From Proposition \ref{prop:2.1}, we have
\begin{equation}\label{eqn:F1}
F^i_jh^j_i=F\left(\frac{(S_k)^i_j h^j_i}{S_k}-\frac{(S_l)^i_j h^j_i}{S_l}\right)=(k-l)F,
\end{equation}
and
\begin{equation}\label{eqn:F2}
F^i_jh^s_ih^j_s=F\left(\frac{(S_k)^i_j h^s_ih^j_s}{S_k}-\frac{(S_l)^i_jh^s_ih^j_s}{S_l}\right)=F\left(-(k+1)\frac{S_{k+1}}{S_k}+(l+1)\frac{S_{l+1}}{S_l}\right).
\end{equation}
We also note that $\nabla^mF=0$ since $F$ is constant.
Then we derive that, by \eqref{eqn:2.8},
\begin{equation*}
\begin{aligned}
F^i_j\nabla_i\nabla^jP=&F^i_jg^{lj}P_{li}\\
=&F^i_jg^{lj}(h_{lj}-h^{s}_{i}h_{sl})\langle e_{n+1},E_{n+1}\rangle+F^i_jg^{lj}h^{s}_{l,i}u_s\\
=&F\left(k-l+(k+1)\frac{S_{k+1}}{S_k}+(l+1)\frac{S_{l+1}}{S_l}\right)\langle e_{n+1},E_{n+1}\rangle+F^i_jg^{lj}h^{s}_{i,l}u_s\\
=&F\left(k-l+(k+1)\frac{S_{k+1}}{S_k}-(l+1)\frac{S_{l+1}}{S_l}\right)\langle e_{n+1},E_{n+1}\rangle.
\end{aligned}
\end{equation*}

Note that $\langle e_{n+1},E_{n+1}\rangle=-\frac{1}w<0$, $\frac{S_{k+1}}{S_k}\leq\frac{n-k}{k+1}$ and  $\frac{S_{l+1}}{S_l}\geq\frac{n-l}{l+1}$ in Lemma \ref{lem:2.4}. We can obtain $F^i_jP^j_i\geq0$

Since $F^i_j\nabla_i\nabla^j$ is elliptic, by the strong maximum principle and boundary conditions, we obtain the assertion.
\end{proof}

Then, we establish an integral equality on $\Omega$.
\begin{Lemma}\label{lem:3.3}
If $u$ is a solution to \eqref{eqn:3.1} and $\eqref{eqn:3.2}$, then
\begin{equation}
\begin{aligned}
0=&\int_{\Omega}\left(k\binom{n}{l}S_k-l\binom{n}{k}S_l\right)(u-c)\text{d}x\\
&+\int_{\Omega}\left((n-k+1)\binom{n}{l}S_{k-1}-(n-l+1)\binom{n}{k}S_{l-1}\right)\left(\langle e_{n+1},E_{n+1}\rangle+\frac{1}{\sqrt{1-c_2^2}}\right) \text{d}x.
\end{aligned}
\end{equation}

\end{Lemma}

\begin{proof}
By the Proposition \ref{prop:2.1}, we derive that
\begin{equation}\label{eqn:3.5}
\begin{aligned}
&(n-k+1)\int_{\Omega}S_{k-1}\langle e_{n+1},E_{n+1}\rangle \text{d}x=\int_{\Omega}(S_k)^i_j\frac{\partial x_j}{\partial x_i}\langle e_{n+1},E_{n+1}\rangle \text{d}x\\
=&\int_{\Omega}\frac{\partial}{\partial x_i}\left((S_k)^i_jx_j\langle e_{n+1},E_{n+1}\rangle\right)\text{d}x-\int_{\Omega}(S_k)^i_jx_j\frac{\partial }{\partial x_i}\langle e_{n+1},E_{n+1}\rangle \text{d}x.
\end{aligned}
\end{equation}

On the one hand, by the boundary condition $\langle e_{n+1},E_{n+1}\rangle=-\frac1{\sqrt{1-c_2^2}}$ and
the divergence theorem, we obtain
\begin{equation}\label{eqn:3.6}
\begin{aligned}
&\int_{\Omega}\frac{\partial}{\partial x_i}\left((S_k)^i_jx_j\langle e_{n+1},E_{n+1}\rangle\right)\text{d}x=-\frac1{\sqrt{1-c_2^2}}\int_{\partial\Omega}(S_k)^i_jx_j\nu_i \text{d}S\\
=&-\frac1{\sqrt{1-c_2^2}}\int_{\partial\Omega}\frac{\partial }{\partial x_i}((S_k)^i_jx_j)\text{d}x
=-\frac{n-k+1}{\sqrt{1-c_2^2}}\int_{\Omega}S_{k-1} \text{d}x.
\end{aligned}
\end{equation}

On the other hand, by the boundary condition $u=c$, divergence theorem and \eqref{eqn:2.7}, we have
\begin{equation}\label{eqn:3.7}
\begin{aligned}
\int_{\Omega}(S_k)^i_jx_j\frac{\partial }{\partial x_i}\langle e_{n+1},E_{n+1}\rangle \text{d}x=&-\int_{\Omega}(S_k)^i_jx_jh^s_iu_s \text{d}x\\
=&-\int_{\Omega}(S_k)^s_ix_jh^i_j\frac{\partial }{\partial x_s}(u-c) \text{d}x\\
=&-\int_{\Omega}\frac{\partial }{\partial x_s}\left((S_k)^s_ix_jh^i_j(u-c)\right)\text{d}x
+\int_{\Omega} kS_k(u-c)\text{d}x\\
&+\int_{\Omega}(S_k)^s_ix_j\frac{\partial }{\partial x_s}h^i_j(u-c) \text{d}x\\
=&\int_{\Omega} kS_k(u-c)\text{d}x+\int_{\Omega}x_j\frac{\partial }{\partial x_j}S_{k}(u-c) \text{d}x.
\end{aligned}
\end{equation}

Substituting \eqref{eqn:3.6} and \eqref{eqn:3.7} into \eqref{eqn:3.5} give us
\begin{equation}\label{eqn:3.8}
\begin{aligned}
0=&(n-k+1)\int_{\Omega}S_{k-1}\left(\langle e_{n+1},E_{n+1}\rangle+\frac{1}{\sqrt{1-c_2^2}}\right) \text{d}x\\
&+\int_{\Omega}kS_k(u-c)\text{d}x+\int_{\Omega}x_j\frac{\partial}{\partial x_j}S_k(u-c)\text{d}x.
\end{aligned}
\end{equation}

Similarly, we can obtain
\begin{equation}\label{eqn:3.9}
\begin{aligned}
0=&(n-l+1)\int_{\Omega}S_{l-1}\left(\langle e_{n+1},E_{n+1}\rangle+\frac{1}{\sqrt{1-c_2^2}}\right) \text{d}x\\
&+\int_{\Omega}lS_l(u-c)\text{d}x+\int_{\Omega}x_j\frac{\partial}{\partial x_j}S_l(u-c)\text{d}x.
\end{aligned}
\end{equation}
Since $\binom{n}{k}S_l=\binom{n}{l}S_k$, we have $$\binom{n}{k}\frac{\partial}{\partial x_j}S_l=\binom{n}{l}\frac{\partial}{\partial x_j}S_k.$$
This combining with \eqref{eqn:3.8} and \eqref{eqn:3.9} yield to the assertion.
\end{proof}

Finally, we will prove the Theorem \ref{main1}. Let $$Q:=(n-k+1)\binom{n}{l}S_{k-1}-(n-l+1)\binom{n}{k}S_{l-1},$$
and
$$
M=k\binom{n}{l}S_k-l\binom{n}{k}S_l=k\binom{n}{l}S_k\left(1-\frac{l\binom{n}{k}S_l}{k\binom{n}{l}S_k}\right)
=(k-l)\binom{n}{l}S_k>0.
$$
Note that
$$
\frac{S_{k-1}/\binom{n}{k-1}}{S_{l-1}/\binom{n}{l-1}}\geq
\left(\frac{S_{k}/\binom{n}{k}}{S_{l}/\binom{n}{l}}\right)^{\frac{k-1}k}=1.
$$
Then we can see that
\begin{equation}
\begin{aligned}
Q=&(n-k+1)\binom{n}{l}S_{k-1}\left(1-\frac{(n-l+1)\binom{n}{k}S_{l-1}}{(n-k+1)\binom{n}{l}S_{k-1}}\right)\\
=&(n-k+1)\binom{n}{l}S_{k-1}\left(1-\frac{l\binom{n}{k-1}S_{l-1}}{k\binom{n}{l-1}S_{k-1}}\right)\\
\geq&(n-k+1)S_{k-1}\left(1-\frac{l}{k}\right)\\
\geq&(k-l)\binom{n}{l}\frac{(n-k+1)}{k}S_{k-1}.
\end{aligned}
\end{equation}
From Lemma \ref{lem:2.4}, we have $\frac{S_{k-1}}{S_k}\geq\frac{k}{n-k+1}$. Then, we have
$$
M-Q=(k-l)\binom{n}{l}\left(S_k-\frac{(n-k+1)}{k}S_{k-1}\right)\leq0.
$$
From \eqref{eqn:2.6}, we have
\begin{equation}
(S_k)^i_j\nabla_i\nabla^ju=-(S_k)^i_j\langle \phi_{i}^j,E_{n+1}\rangle= -(S_k)^i_jh_i^j\langle e_{n+1},E_{n+1}\rangle=-kS_k\langle e_{n+1},E_{n+1}\rangle>0.
\end{equation}
By the maximum principle, we know that $u<c$ in $\Omega$.
Thus, we obtain from Lemma \ref{lem:3.3} that
\begin{equation*}
\begin{aligned}
&\int_{\Omega}\left(P+c-\frac1{\sqrt{1-c_2^2}}\right)Q\text{d}x\\
=&-\int_{\Omega}(u-c)Q\text{d}x-\int_{\Omega}\left(\langle e_{n+1},E_{n+1}\rangle+\frac{1}{\sqrt{1-c_2^2}}\right)Q\text{d}x\\
=&\int_{\Omega}(u-c)(M-Q)\text{d}x\geq0,
\end{aligned}
\end{equation*}
which contradicts $P<-c+\frac1{\sqrt{1-c_2^2}}$ in $\Omega$.

Hence $P\equiv-c+\frac1{\sqrt{1-c_2^2}}$ in $\bar\Omega$  which implies $\left(\frac{u_j}{\sqrt{1-|Du|^2}}\right)_i=h^{j}_i=\delta_{ij}$. As
a result, $M$ is a part of the hyperboloid, $\Omega=B_R(a)$ is a ball and
$$u=c+\theta_0+\sqrt{1+|x-a|^2},$$
where $R=\sqrt{\theta_0^2-1}$ and fixed $a\in\mathbb R^n$.

This complete the proof of Theorem \ref{main1}.

\section{The proof of Theorem \ref{main2}}~~~~
In this section, we will prove the Theorem \ref{main2}.

From the assumption of Theorem \ref{main2}, By a translation and rescaling we may assume, without loss of generality, that
$$\frac{H_k}{H_l}=\frac{\binom{n}{k}}{\binom{n}{l}}.$$
Then, from \eqref{eqn:2.12} we have
\begin{equation}\label{eqn:4.1}
\frac{S_{k}(A)}{S_{l}(A)}=\frac{S_{n-k}(B)}{S_{n-l}(B)}=\frac{\binom{n}{k}}{\binom{n}{l}} \ \ {\rm in}\ M
\end{equation}
and  $\partial M$ is in a hyperboloid or lightcone
\begin{equation*}
\bar\Sigma=\{x\in\mathbb R^{n,1}|\langle x,x\rangle=c\leq0 \ {\rm and }\ x_{n+1}>0\}.
\end{equation*}

This gives $\langle \phi,\phi\rangle=c$ on $\partial M$. Then the condition that $\theta$ is constant means
$\langle \phi,e_{n+1}\rangle=c_2$ on $\partial  M$.

Assumption that $M$ is on $\bar\Sigma$ implies that
\begin{equation*}
\langle \phi,\phi\rangle\leq c\leq0\ {\rm and}\ \phi\in\mathbb R^{n,1}\ {\rm in}\ M.
\end{equation*}

Then Proposition \ref{prop:1} indicates that $\langle \phi,e_{n+1}\rangle<0$ in $M$.

Firstly, we consider an auxiliary function
$$
P=\frac12\langle\phi,\phi\rangle-\langle \phi,e_{n+1}\rangle.
$$
Then we obtain the following Lemma.
\begin{Lemma}\label{lem:4.1}
Let $u$ be a solution of \eqref{eqn:4.1}. Then
$$
F^{j}_iP_j^i\geq0.
$$
\begin{equation*}
P\equiv\frac12c-c_2 \ \ \  {\rm in}\ M\cup\partial M,
\end{equation*}
or
$$
P<\frac12c-c_2 \ \ \  {\rm in}\ M.
$$
\end{Lemma}
\begin{proof}
Direct calculations show
\begin{equation*}
P_j=\nabla_jP=\langle\phi,\phi_j\rangle-\langle\phi,h^s_j\phi_s\rangle
\end{equation*}
and
\begin{equation*}
\begin{aligned}
P_{ji}=&\nabla_i\nabla_jP=\langle\phi_i,\phi_j\rangle-\langle\phi_i,h^s_j\phi_s\rangle
+\langle\phi,\phi_{ji}\rangle-\langle\phi,h^s_{j,i}\phi_s\rangle-\langle\phi,h^s_{j}\phi_{si}\rangle\\
=&g_{ij}-h_{ij}+h_{ij}\langle \phi, e_{n+1}\rangle-h^s_{j,i}\langle\phi,\phi_s\rangle-h^s_{j}h_{si}\langle\phi,e_{n+1}\rangle.
\end{aligned}
\end{equation*}
By the Proposition \ref{prop:2.1}, we have
\begin{equation}\label{eqn:F3}
F^j_i\delta_j^i=F\left(\frac{(S_k)^i_j \delta^j_i}{S_k}-\frac{(S_l)^i_j \delta^j_i}{S_l}\right)
=F\left(\frac{(n-k+1)S_{k-1}}{S_k}-\frac{(n-l+1)S_{l-1}}{S_l}\right).
\end{equation}
Then, from \eqref{eqn:F1}, \eqref{eqn:F2}, \eqref{eqn:F3} and Codazzi equation, we derive that
 \begin{equation*}
\begin{aligned}
F^j_i\nabla^i\nabla_jP=&F^j_i\delta_j^i-F^j_ih_{j}^i+F^j_ih_{j}^i\langle \phi, e_{n+1}\rangle-F^j_i\nabla^ih^s_{j}\langle\phi,\phi_s\rangle-F^j_ih^s_{j}h_{s}^i\langle\phi,e_{n+1}\rangle\\
=&F\left(\frac{(n-k+1)S_{k-1}}{S_k}-\frac{(n-l+1)S_{l-1}}{S_l}-(k-l)\right)\\
&+\left(k-l+(k+1)\frac{S_{k+1}}{S_k}-(l+1)\frac{S_{l+1}}{S_l}\right)
F\langle\phi,e_{n+1}\rangle-\nabla^sF\langle\phi,\phi_s\rangle.
\end{aligned}
\end{equation*}

Note that $\nabla^sF=0$  and $\langle \phi,e_{n+1}\rangle<0$. From $$\frac{S_{k+1}}{S_k}\leq\frac{n-k}{k+1}, \ \ \frac{S_{l+1}}{S_l}\geq\frac{n-l}{l+1},\ \ \frac{S_{k-1}(A)}{S_k(A)}\geq\frac{k}{n-k+1},\ \ \frac{S_{l-1}(A)}{S_l(A)}\leq\frac{l}{n-l+1}$$ in Lemma \ref{lem:2.4},  we can see that  $F^j_i\nabla^i\nabla_jP\geq0$.
Then the strong maximum principle implies that, either $P\equiv\frac12c-c_2 $ in $M\cup\partial M$ or
$P<\frac12c-c_2 $ in $ M$.

This completes the proof of Lemma \ref{lem:4.1}.
\end{proof}

Define a positive function on a smooth domain $\Omega\subset \Sigma$ by
$$\Phi(z):=z_{n+1}=-\langle z,E_{n+1}\rangle.$$
By a direct calculation shows
\begin{equation*}
D_iD^j\Phi=-\tilde g^{jm}\langle \tilde g_{im}z,E_{n+1}\rangle=\Phi\delta^j_i.
\end{equation*}

Recall $M$ can be parametrized by the Gauss map and the position vector $\phi(z)=Du(z)-u(z)z$. Then we have
\begin{equation}\label{eqn:4.2}
\frac12D_i\langle \phi,\phi\rangle=\langle b^l_i\partial_lz,Du-uz\rangle=b^m_iu_m.
\end{equation}

Then, Gao \cite{Gao2} establish an integral equality on an open subset of the hyperboloid if $\langle \phi,\phi\rangle$ and $u$ are constant on the boundary.

\begin{Lemma}\label{lem:4.2} (Lemma 4 in \cite{Gao2})
If $\langle \phi,\phi\rangle=c$ and $u=c_2$ on $\partial\Omega$ for $c,c_2\in\mathbb R^n$, then, for $l< k$,
\begin{equation*}
\begin{aligned}
\frac{k+1}2\int_{\Omega}(\langle \phi,\phi\rangle-c)S_{n-k-1}(B)\Phi \text{d}\tilde\mu=&(n-k)\int_{\Omega}(u-c_2)S_{n-k}(B)\Phi \text{d}\tilde\mu\\
&+\int_{\Omega}(u-c_2)\tilde g(D(S_{n-k}(B)),D\Phi)\text{d}\tilde\mu,
\end{aligned}
\end{equation*}
and
\begin{equation*}
\begin{aligned}
\frac{l+1}2\int_{\Omega}(\langle \phi,\phi\rangle-c)S_{n-l-1}(B)\Phi \text{d}\tilde\mu=&(n-l)\int_{\Omega}(u-c_2)S_{n-k}(B)\Phi \text{d}\tilde\mu\\
&+\int_{\Omega}(u-c_2)\tilde g(D(S_{n-l}(B)),D\Phi)\text{d}\tilde\mu,
\end{aligned}
\end{equation*}
where $d\tilde\mu$ is the volume form of $\Sigma$.
\end{Lemma}

Next, we also give an integral equality on $M$, which is the key to the
proof of Theorem \ref{main2}.
\begin{Lemma}\label{lem:4.3}
Let $\frac{S_{k}(A)}{S_{l}(A)}=\frac{S_{n-k}(B)}{S_{n-l}(B)}=\frac{\binom{n}{k}}{\binom{n}{l}}$ in $M$, $\langle\phi,\phi\rangle=c$ and $\langle \phi,e_{n+1}\rangle=c_2$ on $\partial M$, then we have
\begin{equation*}
\begin{aligned}
&\int_{M}(\langle \phi,\phi\rangle-c)\left(\binom{n}{l}\frac{k+1}2S_{k+1}(A)-\binom{n}{k}\frac{l+1}2S_{l+1}(A)\right)\langle e_{n+1},E_{n+1}\rangle \text{d}\mu\\
=&\int_{M}(\langle \phi,e_{n+1}\rangle-c_2)\left(\binom{n}{l}(n-k)S_{k}(A)-\binom{n}{k}(n-l)S_{l}(A)\right)\langle e_{n+1},E_{n+1}\rangle \text{d}\mu,
\end{aligned}
\end{equation*}
where $d\mu$ is volume form of $M$.
\end{Lemma}

\begin{proof}
Since $\frac{S_{n-k}(B)}{S_{n-l}(B)}=\frac{\binom{n}{k}}{\binom{n}{l}}$, we have $$
\binom{n}{k}D S_{n-l}(B)=\binom{n}{l}DS_{n-k}(B).$$
This combining with Lemma \ref{lem:4.2} yield to
\begin{equation*}
\begin{aligned}
&\int_{\Omega}(\langle \phi,\phi\rangle-c)\left(\binom{n}{l}\frac{k+1}2S_{n-k-1}(B)-\binom{n}{k}\frac{l+1}2S_{n-l-1}(B)\right)\Phi \text{d}\tilde\mu\\
=&\int_{\Omega}(u-c_2)\left(\binom{n}{l}(n-k)S_{n-k}(B)-\binom{n}{k}(n-l)S_{n-l}(B)\right)\Phi \text{d}\tilde\mu.
\end{aligned}
\end{equation*}

Now, we take $\Omega=e_{n+1}(M)$. By pulling back the integrals to $M$ via the Gauss map, we know
\begin{equation*}
\begin{aligned}
&\int_{M}(\langle \phi,\phi\rangle-c)\left(\binom{n}{l}\frac{k+1}2S_{k+1}(A)-\binom{n}{k}\frac{l+1}2S_{l+1}(A)\right)\langle e_{n+1},E_{n+1}\rangle \text{d}\mu\\
=&\int_{M}(\langle \phi,e_{n+1}\rangle-c_2)\left(\binom{n}{l}(n-k)S_{k}(A)-\binom{n}{k}(n-l)S_{l}(A)\right)\langle e_{n+1},E_{n+1}\rangle \text{d}\mu.
\end{aligned}
\end{equation*}

This completes the proof of Lemma \ref{lem:4.3}.
\end{proof}

Finally, we prove the Theorem \ref{main2}. Let
$$
Q=\binom{n}{l}(k+1)S_{k+1}(A)-\binom{n}{k}(l+1)S_{l+1}(A)
$$
and
$$
\tilde Q=\binom{n}{l}(n-k)S_{k}(A)-\binom{n}{k}(n-l)S_{l}(A)<0.
$$
Since $\frac{S_{k+1}(A)}{S_k(A)}\leq\frac{n-k}{k+1}$ and $\frac{S_{l+1}(A)}{S_l(A)}\geq\frac{n-l}{l+1}$ in Lemma \ref{lem:2.4}, we obtain
\begin{equation*}
\begin{aligned}
\tilde Q-Q=\binom{n}{l}[(n-k)S_{k}(A)-(k+1)S_{k+1}(A)]+\binom{n}{k}[(l+1)S_{l+1}(A)-(n-l)S_{l}(A)]\geq0.
\end{aligned}
\end{equation*}
From Lemma \ref{lem:4.3}, we derive that
\begin{equation*}
\begin{aligned}
\int_{M}(P-\frac12c+c_2)\tilde Q\langle e_{n+1},E_{n+1}\rangle \text{d}\mu=&\int_{M}\left(\frac12(\langle \phi,\phi\rangle-c)-(\langle \phi,e_{n+1}\rangle-c_2)\right)\tilde Q\langle e_{n+1},E_{n+1}\rangle \text{d}\mu\\
=&\int_{M}\frac12\left(\langle \phi,\phi\rangle-c\right)(\tilde Q-Q)\langle e_{n+1},E_{n+1}\rangle \text{d}\mu.
\end{aligned}
\end{equation*}

Note that $\langle \phi,\phi\rangle-c\leq0$ and $\langle e_{n+1},E_{n+1}\rangle<0$. Thus we have
$$
\int_{M}(P-\frac12c+c_2)\tilde Q\langle e_{n+1},E_{n+1}\rangle \text{d}\mu\geq0,
$$
which contracts $P<\frac12c-c_2$ in $M$.

Hence $P\equiv\frac12c-c_2$. As a result, principal curvatures $\lambda_1=\cdots=\lambda_n=1$ everywhere in $M$ which implies that $M$ is hyperboloid.

This completes the proof of Theorem \ref{main2}.

\section{The proof of Theorem \ref{thm:1.5}}
\subsection{Integral identities}~~~~
Let $M$ be the graph of a smooth function $x_{n+1} = u(x)$, and suppose that $M$ is a connected space-like hypersurface with boundary $\partial M$ in $\mathbb R^{n,1}$. If $\partial M$ is on the hyperplane $\mathbb R^n\times\{c\}$ and
$H$ is constant. Without loss of generality, we assume that
$H=n$.
Then we have
\begin{equation}\label{eqn:5.1}
D_i\left(\frac{u_i}w\right)=n\ \ \ \ {\rm in}\ \Omega,
\end{equation}
and
\begin{equation}\label{eqn:5.2}
u=c \ \ \ {\rm on}\ \partial\Omega,
\end{equation}
where $c$ is  constant.
Then, we can choose an orthogonal basis  $\{e_{\alpha}\}_1^{n-1}$ such that
$$
u_{\alpha\beta}=\tilde h_{\alpha\beta}|Du|=\lambda_{\alpha}\delta_{\alpha\beta},\ \  {\rm and }\ e_n=\frac{Du}{|Du|}, \ \ \ {\rm for }\ \alpha,\beta=1,\ldots,n-1,
$$
where  $\tilde h$ the second fundamental form of $\partial\Omega$.
It follows from \eqref{eqn:5.1} that
\begin{equation}\label{eqn:5.3}
n=\frac{u_{nn}+(n-1)H_{\partial\Omega}|Du|}w+\frac{u_{nn}}{w^3}|Du|^2=\frac{u_{nn}}{w^3}+\frac{(n-1)H_{\partial\Omega}|Du|}w
\end{equation}
on $\partial\Omega$, where $H_{\partial\Omega}=\frac{\sum\tilde h_{\alpha\alpha}}{n-1}$.

Since $\Delta P=\frac1w\partial_i(wg^{ij}\partial_jP)$ and $dV_M=wdx$,
then, by using  divergence theorem, we can see that
$$
\int_{M}\Delta PdV_M=\int_{\Omega}\partial_i(wg^{ij}\partial_jP)\text{d}x=\int_{\partial\Omega}wg^{ij}P_j\nu_jd\sigma_x,
$$
where $\nu$ is outward normal vector of $\partial\Omega$
 
Let $k=1$, $l=0$ in Lemma \ref{lem:3.1},  we have, by $P=-u+\frac1w$, \eqref{eqn:5.2}-\eqref{eqn:5.3} and above facts,
\begin{equation}\label{eqn:5.4}
\begin{aligned}
\int_{M}\left(\frac{n-1}nS_1^2-2S_2\right)\text{d}V_M\leq&\int_{M}\Delta PdV_M=\int_{M}\left(\frac{n-1}nS_1^2-2S_2\right)\frac1w\text{d}V_M\\
=&\int_{\Omega}\left(\frac{n-1}nS_1^2-2S_2\right)\text{d}x\\
=&\int_{\partial\Omega} \left( -|Du|+\frac{u_{nn}|Du|}{w^3}\right)\frac1w\text{d}\sigma_x\\
=&\int_{\partial\Omega} (n-1)\left(1-\frac{|Du|}wH_{\partial\Omega}\right)\frac{|Du|}{w}\text{d}\sigma_x.
\end{aligned}
\end{equation}

The fundamental identity \eqref{eqn:5.4} can be re-arranged at least into two ways to yield the Soap Bubble type Theorem.

\begin{Theorem}\label{thm:5.1}
Under  assumption of  Theorem \ref{thm:1.5}, then
 $u$ satisfies \eqref{eqn:5.1} and \eqref{eqn:5.2}. Let
\begin{equation*}
R_0=\frac{n|\Omega|}{|\partial\Omega|},\ \ \ H_0=\frac 1{R_{0}}.
\end{equation*}
Then, the following identity holds:
\begin{equation}\label{eqn:5.5}
\begin{aligned}
&\frac{1}{n-1}\int_{\Omega}\left(\frac{n-1}nS_1^2-2S_2\right)\text{d}x+\frac1{R_0}\int_{\partial\Omega}\left(\frac{|Du|}{w}-R_0\right)^2\text{d}\sigma_x\\
=&\int_{\partial\Omega}(H_0 -H_{\partial\Omega})\left(\frac{|Du|}{w}\right)^2\text{d}\sigma_x.
\end{aligned}
\end{equation}
Furthermore, if the mean curvature $H_{\partial\Omega}$ of $\partial\Omega$ satisfies the inequality $H_{\partial\Omega}\geq H_0$ on $\partial\Omega$, then $\Omega$ is a ball.
In particular the same conclusion holds if $H_{\partial\Omega}$ equals some constant on $\partial\Omega$.
\end{Theorem}

\begin{proof}
It follows from \eqref{eqn:5.1} and \eqref{eqn:5.2} that
\begin{equation*}
\int_{\partial\Omega}\frac{|Du|}{w}\text{d}x=n|\Omega|.
\end{equation*}
Then, we have
$$
\frac1{R_0}\int_{\partial\Omega}\left(\frac{|Du|}{w}\right)^2\text{d}\sigma_x
=\frac{1}{R_0}\int_{\partial\Omega}\left(\frac{|Du|}{w}-R_0\right)^2\text{d}\sigma_x+n|\Omega|
$$
and
\begin{equation*}
\int_{\partial\Omega}H_{\partial\Omega}\left(\frac{|Du|}{w}\right)^2\text{d}\sigma_x=\frac1{R_0}\int_{\partial\Omega}\left(\frac{|Du|}{w}-R_0\right)^2\text{d}\sigma_x
+n|\Omega|+\int_{\partial\Omega}(H_{\partial\Omega}-H_0)\left(\frac{|Du|}{w}\right)^2\text{d}\sigma_x.
\end{equation*}
Substituting them into \eqref{eqn:5.4} yield to \eqref{eqn:5.5}.

If $H_{\partial\Omega}\geq H_0$ on $\partial\Omega$, then the right-hand side in \eqref{eqn:5.5} is non-positive and hence both summands at the left-hand side must be zero, being non-negative. Note in passing that this fact implies that the second summand is zero. Hence $\frac{|Du|}{\sqrt{1-|Du|^2}}=R_0$ which implies that $|Du|$ is constant. Thus  $u$ satisfies \eqref{eqn:3.1} and \eqref{eqn:3.2} for $k=1,l=0$.
\end{proof}

Next, we will show that \eqref{eqn:5.4} can be rearranged into an identity that implies the Heintze-Karcher's inequality.

\begin{Theorem}\label{thm:5.2}
Under  assumption of  Theorem \ref{thm:1.5}, then
 $u$ satisfies \eqref{eqn:5.1} and \eqref{eqn:5.2}. If $\partial\Omega$ is strictly mean-convex, then we have the following identity:
\begin{equation}\label{eqn:5.6}
\begin{aligned}
&\frac{1}{n-1}\int_{\Omega}\left(\frac{n-1}nS_1^2-2S_2\right)\text{d}x+\int_{\partial\Omega}\frac{\left(1-\frac{|Du|}{w}H_{\partial\Omega}\right)^2}{H_{\partial\Omega}}\text{d}\sigma_x\\
=&\int_{\partial\Omega}\frac1{H_{\partial\Omega}}\text{d}\sigma_x-n|\Omega|.
\end{aligned}
\end{equation}
In particular, the Heintze-Karcher's inequality
\begin{equation}\label{eqn:5.7}
\int_{\partial\Omega}\frac1{H_{\partial\Omega}}\text{d}\sigma_x\geq n|\Omega|.
\end{equation}
Moreover, the equality of \eqref{eqn:5.7} holds if and only if $\Omega$ is a ball. If $H_{\partial\Omega}$ is constant, then $\partial\Omega$ is a sphere.
\end{Theorem}

\begin{proof}
By integrating on $\partial\Omega$ the identity
$$
\frac{\left(1-\frac{|Du|}{w}H_{\partial\Omega}\right)^2}{H_{\partial\Omega}}
=-\left(1-\frac{|Du|}{w}H_{\partial\Omega}\right)\frac{|Du|}{w}+\frac1{H_{\partial\Omega}}-\frac{|Du|}{w}.
$$
Substituting it into \eqref{eqn:5.4} yield to \eqref{eqn:5.6}.

Both summands at left-hand side of \eqref{eqn:5.6} are non-negative, then \eqref{eqn:5.7} follows. If the right-hand side is zero, those summands must be zero. The vanishing of the first summand implies that $\Omega$ is a ball, as already noticed. Note in passing that the vanishing of the second summand gives that $\frac{1}{H_{\partial\Omega}}=\frac{|Du|}{w}$ on $\partial\Omega$, which also implies radial symmetry, by Theorem \ref{thm:5.1}. Finally, if $H_{\partial\Omega}$ equals some constant on $\partial\Omega$, we know that such a constant must have the value $H_0$ in Theorem \ref{thm:5.1}, that implies that the right-hand side of \eqref{eqn:5.6} is null and hence, once again, $\Omega$ must be a ball.
\end{proof}

Finally, combining with Theorem \ref{thm:5.1}, Theorem \ref{thm:5.2} and Theorem \ref{main1} yield to Theorem \ref{thm:1.5}.

\subsection{Stability of space-like hypersurface}~~~~
We collect our results on the stability of the spherical configuration
by putting together the identities \eqref{eqn:5.5} and \eqref{eqn:5.6} and  the estimates obtained by Bartnik-Simon \cite{BS}.

On the one hand, we know that
\begin{equation}\label{eqn:5.8}
-\langle e_{n+1},E_{n+1}\rangle=\sqrt{\frac1{1-|Du|^2}}>1.
\end{equation}
Bartnik-Simon proved that $\sup_{\partial\Omega}|Du|\leq 1-\Theta(n,\Omega)<1$ in Corollary 3.4 of \cite{BS}. It follows that
\begin{equation}\label{eqn:5.9}
\sup_{\partial\Omega}\frac{|Du|}{\sqrt{1-|Du|^2}}\leq\frac{1-\Theta}{\sqrt{2\Theta-\Theta^2}}:=K.
\end{equation}

On the other hand, we have
\begin{equation}\label{eqn:5.10}
|\bar h|^2=\frac{n-1}nH^2-2H_2,
\end{equation}
where  $\bar h$ is the trace-free second fundamental of $M$ defined by
$
\bar h_{i}^j=h_{i}^j-\frac1nH\delta^j_i.
$

Let 
$$
\|\bar h\|^2_{L^2(M)}=\int_{M}|\bar h|^2dV_M, \ \ \|\bar h\|^2_{L^2(\Omega)}=\int_{\Omega}|\bar h|^2dx
$$

Combining \eqref{eqn:5.8}-\eqref{eqn:5.10} with integral identities \eqref{eqn:5.5} and \eqref{eqn:5.6} yield to the following theorems.

\begin{Theorem}\label{thm:5.3}
Under  assumption of  Theorem \ref{thm:1.5}, then
 $u$ satisfies \eqref{eqn:5.1} and \eqref{eqn:5.2}. The following estimation holds:
\begin{equation*}
\begin{aligned}
\|\bar h\|_{L^2(M)}\leq\|\bar h\|_{L^2(\Omega)}
\leq C(n,K)\|H_0 -H_{\partial\Omega}\|_{L^1(\partial\Omega)}^{1/2},
\end{aligned}
\end{equation*}
where $C(n,K)=\sqrt{n-1}K(\Theta)$ is constant and $H_0=\frac{|\partial\Omega|}{n|\Omega|}$.
\end{Theorem}

We also note that
$$
\int_{\partial\Omega}\frac1{H_{\partial\Omega}}\text{d}\sigma_x-n|\Omega|
=\int_{\partial\Omega}\frac1{H_{\partial\Omega}}-\frac1{H_0}\text{d}Sx\leq\frac{n|\Omega|}{|\partial\Omega|}\|H_0 -H_{\partial\Omega}\|_{\infty,\partial\Omega}\int_{\partial\Omega}\frac1{H_{\partial\Omega}}\text{d}\sigma_x,
$$
and
$$
\frac{|Du|}{\sqrt{1-|Du|^2}}=\sqrt{\theta^2-1}.
$$
Then, we obtain the following estimations.
\begin{Theorem}\label{thm:5.4}
Under  assumption of  Theorem \ref{thm:1.5}, then
 $u$ satisfies \eqref{eqn:5.1} and \eqref{eqn:5.2}. Suppose that there exists a constant $\underline{H}>0$  such that $H_{\partial\Omega}\geq\underline{H}$. Then the following estimations hold:
\begin{equation*}
\begin{aligned}
\|\bar h\|_{L^2(M)}\leq\|\bar h\|_{L^2(\Omega)}
\leq \left((n-1)\int_{\partial\Omega}\frac1{H_{\partial\Omega}}\text{d}\sigma_x-n|\Omega|\right)^{1/2},
\end{aligned}
\end{equation*}
\begin{equation*}
\begin{aligned}
\|\bar h\|_{L^2(M)}\leq\|\bar h\|_{L^2(\Omega)}
\leq \sqrt{n-1}\|1/{H_{\partial\Omega}}-\sqrt{\theta^2-1}\|_{L^1(\partial\Omega)}^{1/2},
\end{aligned}
\end{equation*}
or
\begin{equation*}
\begin{aligned}
\|\bar h\|_{L^2(M)}\leq\|\bar h\|_{L^2(\Omega)}
\leq \sqrt{\frac{n(n-1)|\Omega|}{\underline{H}}}\|H_0 -H_{\partial\Omega}\|_{\infty,\partial\Omega}^{1/2},
\end{aligned}
\end{equation*}
where $H_0=\frac{|\partial\Omega|}{n|\Omega|}$.
\end{Theorem}

Next, we will give lower bound estimation of $\|\bar h\|^2_{L^2(\Omega)}$. We begin to collect some facts.  Without loss of generality, we assume that $c=0$ in \eqref{eqn:5.1}.
Since $\partial\Omega$ is bounded and of class $C^{2,\alpha}$, it has the properties of the
uniform interior and exterior sphere condition, whose respective radii will be designated by $r_i$ and $r_e$. We recall that
for a point $z\in\Omega$, $\rho_i$ and $\rho_e$ denote the radius of the largest
ball centered at $z$ and contained in $\Omega$ and that of the smallest ball that contains $\Omega$
with the same center, as defined by 
$$
\rho_i=\min_{x\in\partial\Omega}|x-z|\ \ {\rm and}\ \rho_e=\max_{x\in\partial\Omega}|x-z|.
$$

Let 
\begin{equation}\label{q}
M=\max_{\bar\Omega}|\nabla u|=\max_{\partial\Omega}u_{\nu}\ {\rm and}\ q=\sqrt{1+|x-z|^2}-a.
\end{equation}

Let $h=q-u$ and  $d_{\Omega}$ be an diameter of $\Omega$.Then, we know that
$h=q$ on $\partial\Omega$ and
\begin{equation}\label{eqn:05.11}
\max_{\partial\Omega}h-\min_{\partial\Omega}h=\max_{\partial\Omega}q-\min_{\partial\Omega}q=\sqrt{1+\rho_e^2}-\sqrt{1+\rho_i^2}
\geq\frac{(\rho_e-\rho_i)(\rho_e+\rho_i)}{2\sqrt{1+d_{\Omega}^2}}. 
\end{equation}
Since $q$ and $u$  satisfy ${\rm div}(A(Dw))=n$, where $A(p)=p/{\sqrt{1-|p|^2}}$, we can see that
$$
{\rm div}(\Phi(x)Dh)=0,
$$
where 
$$
\Phi(x)=\int_{0}^1 DA(Du+t(Dq-Du))dt. 
$$
On the other hand, since $\Omega$ is bounded domain, we have $$|Dq|=\frac{|x-z|}{\sqrt{1+|x-z|}}\leq\frac{R}{\sqrt{1+R^2}}<1,$$
where $R=\max_{x\in\Omega}|x-z|$.

 Note that, in Theorem 3.6 of \cite{BS}, 
$$
|Du|\leq 1-\Theta(n,\Omega)\ \  {\rm in}\ \Omega.
$$
By directing computation, we know that $\Phi(x)$ is an uniform elliptic matrix i.e.  one holds that
$$
\lambda|\xi|^2\leq\xi^T\Phi(x)\xi\leq \Lambda|\xi|^2 \ \ {\rm for}\ \xi\in\mathbb R^n,
$$
where $\lambda=1$ and $\Lambda=\frac{1}{(2\delta-\delta^2)^{3/2}}$ for $\delta=\min\{\Theta,\frac{\sqrt{1+R^2}-R}{\sqrt{1+R^2}}\}$. Since $R\leq d_{\Omega}$, we can see that $\Lambda$ depends on $n$, $\Omega$ and $d_{\Omega}$.

To realize our goal, we start by proving  Poincar\'e-type inequalities.

\begin{Lemma}[Corollary 3.2 in \cite{MP2024}]\label{rem:5.6}
Let $\Omega \subset \mathbb{R}^n$ be an connected Lipschitz domain. If $f \in W^{2,2}(\Omega)$ satisfies $\int_{\Omega}Dfdx=\overrightarrow{0}$, then there exists a positive  constant $\mu_{1}(\Omega)$ such
that 
\begin{equation*}
\int_{\Omega}|Df|^2dx\leq \mu_{1}(\Omega)\int_{\Omega}|D^2f|^2dx.
\end{equation*}
\end{Lemma}
\begin{Lemma}[Lemma 3.4 in \cite{MP2024}]\label{rem:5.7}
Let $\Omega \subset \mathbb{R}^n$ be an connected bounded domain satisfying uniform interior sphere. If $f \in W^{1,\infty}(\Omega)$, there holds
\begin{equation*}
\max_{\bar\Omega}f-\min_{\bar\Omega}f\leq C(n,d_{\Omega},r_i)\left\{
\begin{aligned}
&\|Df\|^{\frac2{n+1}}_{L^2}\|Df\|^{\frac{n-1}{n+1}}_{L^{\infty}},\ \ n\geq 3,\\
&\|Df\|_{L^2}\log\left(e|\Omega|^{\frac12}\frac{\|Df\|_{L^{\infty}}}{\|Df\|_{L^2}}\right),\ \ n=2.
\end{aligned}\right.
\end{equation*}
\end{Lemma}
\begin{Lemma}\label{lem:5.5}
Let $\Omega\subset\mathbb R^{n}$ be a connected and bounded  domain with boundary of $C^{2,\alpha}$.
Let $f\in C^{2,\alpha}(\bar\Omega)$. 
 If ${\rm div}(\Phi(x)Df)=0$ and $f(x_0)=0$ where $x_0$ is  given point in $\Omega$, then there exists a positive  constant $\mu_{0}(\Omega,\lambda,\Lambda)$ such
that 
\begin{equation}\label{P1}
\int_{\Omega}|f|^2dx\leq \mu_{0}^{-1}(\Omega,\lambda,\Lambda)\int_{\Omega}|Df|^2dx.
\end{equation}
\end{Lemma}
\begin{proof}
Define
\[
\mu_0(\Omega, \Phi) = \inf \left\{ \int_\Omega \langle \Phi(x) \nabla f, \nabla f \rangle  dx : \int_\Omega f^2  dx = 1,  \mathrm{div}(\Phi(x) \nabla f) = 0 \ \text{in}\  \Omega,  f(x_0) = 0 \right\}.
\]

Assume by contradiction that $\mu_0(\Omega, \Phi) = 0$. Then there exists a sequence $\{f_n\} \subset H^1(\Omega)$ satisfying:\\
 (1) $\mathrm{div}(\Phi(x) \nabla f_n) = 0$ in $\Omega$,
   (2)$f_n(x_0) = 0$,
  (3) $\int_\Omega f_n^2  dx = 1$,
(4) $\int_\Omega \langle \Phi(x) \nabla f_n, \nabla f_n \rangle  dx \to 0$.

Since $\Phi(x)$ is uniformly elliptic, there exists $\lambda > 0$ such that
\[
\langle \Phi(x) \xi, \xi \rangle \geq \lambda |\xi|^2 \quad \text{for all } \xi \in \mathbb{R}^n, x \in \Omega.
\]
Therefore,
\[
\int_\Omega |\nabla f_n|^2  dx \leq \lambda^{-1} \int_\Omega \langle \Phi(x) \nabla f_n, \nabla f_n \rangle  dx \to 0.
\]
Hence, $\|\nabla f_n\|_{L^2(\Omega)} \to 0$. Combined with $\|f_n\|_{L^2(\Omega)} = 1$, the sequence $\{f_n\}$ is bounded in $H^1(\Omega)$.

By the Rellich-Kondrachov theorem (applicable since $\Omega$ is bounded with $C^{2,\alpha}$ boundary, hence Lipschitz), there exists a subsequence (still denoted $\{f_n\}$) and a function $f \in L^2(\Omega)$ such that $f_n \to f$ strongly in $L^2(\Omega)$. 

Since $\|\nabla f_n\|_{L^2} \to 0$ and the $H^1$ norm is weakly lower semicontinuous, we have $\nabla f = 0$ in the distributional sense. As $\Omega$ is connected, $f$ is constant almost everywhere.

Now we prove that $f$ is identically zero. Since $\Phi(x)$ is uniformly elliptic, by the De Giorgi-Nash theorem, there exists $\alpha \in (0,1)$ and a constant $C$ (depending on the ellipticity constants and dimension $n$) such that for any ball $B_r(x_0) \subset \Omega$,
\[
\|f_n\|_{C^{0,\alpha}(B_{r/2}(x_0))} \leq C \|f_n\|_{L^2(B_r(x_0))}.
\]

Since $\{f_n\}$ is bounded in $L^2(\Omega)$, it is also bounded in $L^2(B_r(x_0))$, hence $\{f_n\}$ is bounded in $C^{0,\alpha}(B_{r/2}(x_0))$. By the Arzelà-Ascoli theorem, there exists a subsequence that converges uniformly to some function $g$ on $B_{r/2}(x_0)$. But $f_n \to f$ in $L^2(\Omega)$, so $g = f$ almost everywhere on $B_{r/2}(x_0)$, and by uniform convergence:\\
(1) $f$ is continuous on $B_{r/2}(x_0)$ (as a uniform limit), (2) $f_n(x) \to f(x)$ for all $x \in B_{r/2}(x_0)$,
    (3) In particular, $f(x_0) = \lim_{n\to\infty} f_n(x_0) = 0$.

Since $f$ is continuous on $B_{r/2}(x_0)$ and constant almost everywhere, by continuity $f$ is identically constant a.e. on $B_{r/2}(x_0)$. Combined with $f(x_0) = 0$, we have $f \equiv 0$  on $B_{r/2}(x_0)$.

Now, $\Omega$ is connected and $f$ is constant almost everywhere, while being identically zero on the nonempty open set $B_{r/2}(x_0)$. Therefore, $f \equiv 0$ a.e. in $\Omega$.

However, by $L^2$ convergence, $\int_\Omega f_n^2 dx \to \int_\Omega f^2 dx = 0$, which contradicts $\int_\Omega f_n^2 dx = 1$. Hence, the assumption $\mu_0(\Omega, \Phi) = 0$ is false, and we conclude that $\mu_0(\Omega, \Phi) > 0$.

By using the ellipticity condition
\[
\int_\Omega \langle \Phi(x) \nabla f, \nabla f \rangle dx\leq \Lambda \int_\Omega |\nabla f|^2 dx,
\]
we obtain
\[
\int_{\Omega} |f|^2 dx \leq \mu_0^{-1} \int_\Omega \langle \Phi(x) \nabla f, \nabla f \rangle dx \leq \mu_0^{-1} \Lambda \int_{\Omega} |\nabla f|^2 dx.
\]

Setting $\mu_0(\Omega, \lambda, \Lambda) = \Lambda^{-1} \mu_0(\Omega, \Phi)$ gives the desired inequality
\[
\int_{\Omega} |f|^2 dx \leq \mu_0^{-1}(\Omega, \lambda, \Lambda) \int_{\Omega} |\nabla f|^2 dx.
\]

This completes the proof of Lemma \ref{lem:5.5}.
\end{proof}

Then, we adopt Magnanini-Poggesi's ideal (Lemma 3.3 of \cite{Ma}) to link that oscillation with the $L^2$-norm of $h-h(z)$.
To this aim, we define the parallel set as
$$
\Omega_{\sigma}=\{y\in\Omega:\ {\rm dist}(y,\partial\Omega)>\sigma\}\ \ {\rm for}\ \ 0<\sigma<r_i.
$$

We obtain the following result. without loss of generality, we ways assume that $c=0$ in \eqref{eqn:5.2}.
\begin{Lemma}\label{lem:5.6}
Let $\Omega\subset\mathbb R^n$, $n\geq2$, be a bounded domain with boundary of class $C^{2,\alpha}$. Set $h=q-u$, where $u$ is the solution of  \eqref{eqn:5.1} and \eqref{eqn:5.2} and  $q$ is defined by \eqref{q}. Then, if
\begin{equation}\label{eqn:5.11}
\|h-h(z)\|_{L^2(\Omega)}<\frac{\sqrt{|B|}}{n2^{n+1}}Mr_i^{\frac{n+2}{2}}
\end{equation}
holds, we have that
\begin{equation}\label{eqn:5.12}
\rho_e-\rho_i\leq C(n,d_{\Omega},\Omega)\frac{M^{\frac{n}{n+2}}}{|\Omega|^{\frac1n}}\|h-h(z)\|_{L^2(\Omega)}^{2/(n+2)},
\end{equation}
where
$$
C(n,d_{\Omega},\Omega)=\frac{2^{2+\frac{n}{n+2}}(2C_1(n,d_{\Omega},\Omega)+n)}{n^{\frac{n}{n+2}}}|B|^{\frac1n-\frac1{n+2}}.
$$
\end{Lemma}
\begin{proof}
Let $x_i$ and $x_e$ be points in $\partial\Omega$ that minimize (resp. maximize) $q$ on $\partial\Omega$ and, for 
$0<\sigma<r_i$.
Define $y_j=x_j-\sigma\nu(x_j)$ for $j=i,e$. By directing computation, we obtain
that
$$
h(y_j)=\sqrt{1+(\rho_j-\sigma)^2}-a+\int_{0}^{\sigma}\langle Du(x_j-t\nu_j),\nu_j\rangle dt \ \ {\rm for} \ j=i,e.
$$
Therefore, we have
\begin{equation}\label{eqn:5.13}
h(y_e)-h(y_i)=\sqrt{1+(\rho_e-\sigma)^2}-\sqrt{1+(\rho_i-\sigma)^2}+\Psi,
\end{equation}
where $\Psi=\int_{0}^{\sigma}\langle Du(x_e-t\nu_e),\nu_e\rangle-\langle Du(x_i-t\nu_i),\nu_i\rangle dt$ and $|\Psi|\leq 2M\sigma$.

We also note  that
$$
\sqrt{1+(\rho_e-\sigma)^2}-\sqrt{1+(\rho_i-\sigma)^2}=
\frac{(\rho_e-\rho_i)(\rho_e+\rho_i-2\sigma)}{\sqrt{1+(\rho_e-\sigma)^2}+\sqrt{1+(\rho_i-\sigma)^2}},
$$
and
$$
\sqrt{1+(\rho_e-\sigma)^2}+\sqrt{1+(\rho_i-\sigma)^2}\leq2(1+d_{\Omega})
$$
where $d_{\Omega}$ is diameter of $\Omega$.

Substituting these into \eqref{eqn:5.13} yield to 
$$
(\rho_e-\rho_i)(\rho_e+\rho_i-2\sigma)\leq2(1+d_{\Omega})(|h(y_e)-h(y_i)|+2M\sigma),
$$
for very $\sigma<\min\{\frac{\rho_e+\rho_i}2,r_i\}$.

let $\sigma<\sigma_0$ with 
$$
\sigma_0=\frac14\left(\frac{|\Omega|}{|B|}\right)^{1/n},
$$
we can see that $\rho_e+\rho_i-2\sigma\geq 2\sigma_0$. It follows 
that
\begin{equation}\label{eqn:5.14}
\rho_e-\rho_i\leq \frac{1+d_{\Omega}}{\sigma_0}\left(|h(y_e)-h(y_i)|+2M\sigma\right).
\end{equation}

Since $div(\Phi(x)Dh)=0$ and $\Phi(x)$ is  an uniform elliptic condition, we have
$$
|h(y_j)-h(z)|\leq \frac{C_1}{|B|\sigma^{n}}\int_{B_{\sigma}(y_j)}|h-h(z)|dy\leq\frac{C_1}{\sqrt{|B|\sigma^{n}}}\parallel h-h(z)\parallel_{L^2(\Omega)}
$$
by h\"older inequality, where constant $C_1$ is  the dependent $n$ and elliptic constant.

This combining with \eqref{eqn:5.14} yield that
\begin{equation}\label{eqn:5.15}
\rho_e-\rho_i\leq\frac{2(1+d_{\Omega})}{\sigma_0}\left(\frac{C_1}{\sqrt{|B|\sigma^{n}}}\parallel h-h(z)\parallel_{L^2(\Omega)}+M\sigma\right),
\end{equation}
for very  $0<\sigma<\min\{\sigma_0,r_i\}$.

Therefore, by minimizing the right-hand side of \eqref{eqn:5.14}, we can conveniently choose
$$
\sigma=\left(\frac{n\parallel h-h(z)\parallel_{L^2(\Omega)}}{2|B|^{1/2}M}\right)^{\frac{2}{n+2}}
$$
in \eqref{eqn:5.15} obtains \eqref{eqn:5.12}. If $\sigma<\frac{r_i}4<\min\{\sigma_0,r_i\}$, it follows that \eqref{eqn:5.11} holds.

\end{proof}

Furthermore,
according to Lemma \ref{lem:5.5} and \ref{lem:5.6}, we obtain the  following result.

\begin{Theorem}\label{thm:5.7}
Let $\Omega\subset\mathbb R^n$, $n\geq2$, be a bounded domain with boundary of class $C^{2,\alpha}$ and $u$ be the solution of  \eqref{eqn:5.1} and \eqref{eqn:5.2}. Set $h=q-u$, where  $q$ is defined by \eqref{q} and $z$ is chosen as
$$
z=\frac{1}{|\Omega|}\left(\int_{\Omega}x dx-\int_{\Omega}\frac{Du}{\sqrt{1-|Du|^2}} dx\right).
$$
If $z\in \Omega$, then \\
(1) if $n=2$ or $3$, then there exists a constant $C$ such that
\begin{equation}\label{eqn:5.16}
\rho_e-\rho_i\leq C\|D(v-w)\|_{L^2(\Omega)},\  v=\frac{Du}{\sqrt{1-|Du|^2}},\, w=x-z,
\end{equation}
where
$$
C=2c(n,\Omega)d^{\gamma}_{\Omega}\sqrt{1+d^2_{\Omega}}\sqrt{\mu_1(\mu_0^{-1}+1)+(6\sqrt{ n\mu_1}+1)^2}\left(\frac{|B|}{|\Omega|}\right)^{\frac1n},
$$
and $c(n,\Omega)$ is Sobolev constant.\\
(2) if $n\geq 4$, then there exists two constants $C$ and $\varepsilon$ such that
$$
\|D(v-w)\|_{L^2(\Omega)}<\varepsilon,
$$
then
\begin{equation}\label{eqn:5.17}
\rho_e-\rho_i\leq C\|D(w-v)\|_{L^2(\Omega)}^{\frac{2}{n+2}}
\end{equation}
where
$$
C=C(n,d_{\Omega},\Omega)(\mu_1\mu_0^{-1})^{\frac1{n+2}}\frac{M^{\frac{n}{n+2}}}{|\Omega|^{\frac1n}},\ \ \varepsilon=\frac1{\sqrt{\mu_1\mu_0^{-1}}}\frac{\sqrt{|B|}}{n2^{n+1}}Mr_i^{\frac{n+2}{2}}.
$$

\end{Theorem}
\begin{Remark}
Since we want to using Poincar\'e inequality in Remark \ref{rem:5.6}, we choose $z$ such that $\int_{\Omega}Dhdx=\overrightarrow{0}$. However, we can't prove that
$z$ must be in $\Omega$. So we assume that $z\in\Omega$. Some similarly discussions can be find in Theorem 3.6 in \cite{Ma}. They choose $z$ is the mass centre of $\Omega$, but $z$ may not be in $\Omega$. Therefore they suppose that $z\in\Omega$ or $\Omega$ is convex. 
\end{Remark}
\begin{proof}
{\bf Step 1: The relation between $|D^2h|$ and $|D(v-w)|$.} Let Let $\Omega\subset\mathbb R^{n}$ be a connected and bounded  domain with boundary of $C^{2,\alpha}$. 
Then we have $\nabla u=\frac{v}{\sqrt{1+|v|}}$ and  $\nabla q=\frac{w}{\sqrt{1+|w|}}$, where $v=\frac{Du}{\sqrt{1-|Du|^2}}$ and $w=x-z$. We consider the $$G(y)=\frac{y}{\sqrt{1+|y|^2}}$$ for $y\in\mathbb R^n$. Then $DG(y)=\frac{1}{\sqrt{1+|y|^2}}I-\frac{1}{(1+|y|^2)^{3/2}}y\otimes y$, which implies that
$$
|DG(y)\xi|\leq\frac{|\xi|}{\sqrt{1+|y|^2}}+\frac{|y|^2|\xi|}{(1+|y|^2)^{\frac32}}=\frac{|\xi|}{(1+|y|^2)^{\frac12}}
$$
Therefore, 
$$
|DG(y)|\leq \frac{1}{(1+|y|^2)^{\frac12}}\leq 1.
$$

By directing computation, we obtain that
$$Dh=Dq-Du=G(w)-G(v).$$

It follows that
\begin{equation}\label{eqn:CC}
\|Dh\|_{L^2}\leq\|v-w\|_{L^2}.
\end{equation}

Taking the derivative on both sides, we have
\begin{equation}\label{eqn:5.19}
D^2h=DG(w)Dw-DG(v)Dv=[DG(w)-DG(v)]Dw+DG(v)(Dw-Dv) 
\end{equation}

Next, we proceed to estimate the right-hand side of equation \eqref{eqn:5.19}.
Through detailed computations, it can be shown that $\|D^2G\| \leq 6$.
Moreover, since $|Dw| = \sqrt n$, we derive
$$
|[DG(w)-DG(v)]Dw|\leq\sqrt n|DG(w)-DG(v)|\leq \|D^2G\|\cdot|w-v|\leq 6\sqrt n|w-v|,
$$
and
$$
|DG(v)(Dw-Dv)|\leq \|DG\|\cdot |D(w-v)|\leq |D(w-v)|.
$$

This combining with \eqref{eqn:5.19} yield to
$$
|D^2h|\leq 6\sqrt n|w-v|+|D(w-v)|.
$$

{\bf Step 2: The relation between $\|D^2h\|_{L^2}$ and $\|D(v-w)\|_{L^2}$.}

Since 
$$
z=\frac{1}{|\Omega|}\left(\int_{\Omega}x dx-\int_{\Omega}\frac{Du}{\sqrt{1-|Du|^2}} dx\right),
$$ we can see that 
$$
\int_{\Omega}(w-v)dx=\overrightarrow{0}.
$$
It follows from Lemma \ref{rem:5.6}  that
$$
\|v-w\|_{L^2}^2\leq \mu_1(\Omega)\|D(v-w)\|_{L^2}^2.
$$

It is easy to see that, by using Remark \ref{rem:5.6},
\begin{equation}\label{eqn:5.20}
\begin{aligned}
\|D^2h\|_{L^2}^2\leq& 36n\|v-w\|_{L^2}^2+\|D(v-w)\|_{L^2}^2+12\sqrt n\int_{\Omega}|w-v|\cdot |D(w-v)| dx\\
\leq&(36n\mu_1+1)\|D(v-w)\|_{L^2}^2+12\sqrt n\|w-v\|_{L^2} \|D(w-v)\|_{L^2}\\
\leq &(36n\mu_1+1)\|D(v-w)\|_{L^2}^2+12\sqrt{n\mu_1}\|D(w-v)\|_{L^2}^2\\
=&(6\sqrt{ n\mu_1}+1)^2\|D(v-w)\|_{L^2}^2.
\end{aligned}
\end{equation}
Here we have used h\"older inequality.

{\bf Step 3: Conclusion.}

(i) Let $n=2$ or $3$. By the Sobolev embedding theorem (see Theorem 3.12 of \cite{Gi} for details),  we have that there is a constant $c $ such that, for any $f\in W^{2,2}$, one holds that
$$
\frac{|f(x)-f(y)|}{|x-y|^{\gamma}}\leq c\|f\|_{W^{2,2}(\Omega)}\ \ {\rm for\ any}\ x,y \in\bar\Omega \ {\rm with}\ x\not=y,
$$
where $\gamma$ is any number in $(0,1)$ for $n=2$ and $\gamma=1/2$ for $n=3$.

Let  $f=h-h(z)$,  then $f(z)=0$.  By using \eqref{P1}, Remark \ref{rem:5.6}, \eqref{eqn:CC} and \eqref{eqn:5.20}, we obtain that
\begin{equation*}
\begin{aligned}
\|h-h(z)\|_{W^{2,2}}=&\left(\int_{\Omega}|h-h(z)|^2+|Dh|^2+|D^2h|^2dx\right)^{\frac12}\\
\leq&\left((\mu_0^{-1}+1)\|v-w\|_{L^2}^2+(6\sqrt{ n\mu_1}+1)^2\|D(v-w)\|_{L^2}^2\right)^{\frac12}\\
\leq&\sqrt{\mu_1(\mu_0^{-1}+1)+(6\sqrt{ n\mu_1}+1)^2}\|D(v-w)\|_{L^2}
\end{aligned}
\end{equation*}
Since ${\rm div}(\Phi(x)Dh)=0$, it attains its extrema on $\partial\Omega$ and hence we have that
\begin{equation*}
\begin{aligned}
\sqrt{1+\rho_e^2}-\sqrt{1+\rho_i^2}=&\max_{\partial\Omega}h-\min_{\partial\Omega}h\\
\leq& cd^{\gamma}_{\Omega}\sqrt{\mu_1(\mu_0^{-1}+1)+(6\sqrt{ n\mu_1}+1)^2}\|D(v-w)\|_{L^2}.
\end{aligned}
\end{equation*}
Thus, \eqref{eqn:5.16} follows by $\rho_e+\rho_i\geq\rho_e\geq(|\Omega|/|B|)^{1/n}$.

(ii) Let $n\geq 4$. By the same choice of $f$ as in (i), we obtain that
\begin{equation*}
\begin{aligned}
\|h-h(z)\|_{L^2(\Omega)}\leq&\sqrt{\mu_0^{-1}}\|Dh\|_{L^2}
\leq \sqrt{\mu_0^{-1}}\|w-v\|_{L^2(\Omega)}\leq\sqrt{\mu_1\mu_0^{-1}}\|D(w-v)\|_{L^2}.
\end{aligned}
\end{equation*}
Hence, \eqref{eqn:5.17} follows from Lemma \ref{lem:5.6}.
\end{proof}

Finally, we present a stability estimate that states that a compact hypersurface
$\partial\Omega$ and $M$ can be contained in a spherical annulus whose interior and exterior  radii $\rho_i$ and $\rho_e$.

\begin{Theorem}\label{thm:5.8}
Let $\Omega\subset\mathbb R^n$, $n\geq2$, be a bounded domain with boundary of class $C^{2,\alpha}$.
Let $u$ be a solution of \eqref{eqn:5.1} and \eqref{eqn:5.2}. Let 
$$
z=\frac{1}{|\Omega|}\left(\int_{\Omega}x dx-\int_{\Omega}\frac{Du}{\sqrt{1-|Du|^2}} dx\right).
$$
If $z\in\Omega$, then we have\\
(1) if $n=2,$ or $n=3$, there exists a positive constant $C$ such that
\begin{equation*}
\begin{aligned}
\rho_e-\rho_i\leq C\|H_0 -H_{\partial\Omega}\|_{L^1(\partial\Omega)}^{1/2}
\end{aligned}
\end{equation*}
where 
$$
C=2c\sqrt{n-1}K(\Theta)d^{\gamma}_{\Omega}\sqrt{1+d^2_{\Omega}}\sqrt{\mu_1(\mu_0^{-1}+1)+(6\sqrt{ n\mu_1}+1)^2}\left(\frac{|B|}{|\Omega|}\right)^{\frac1n}.
$$
(2) if $n\geq 4$, then there exists two constants $C$ and $\varepsilon$ such that if
$$
\|H_0-H_{\partial\Omega}\|_{L^1(\partial\Omega)}<\varepsilon,
$$
then
\begin{equation}\label{eqn:5.18}
\rho_e-\rho_i\leq C\|H_0-H_{\partial\Omega}\|_{L^1(\partial\Omega)}^{\frac{1}{n+2}}
\end{equation}
where
$$
C=C(n,d_{\Omega},\Omega)(\mu_1\mu_0^{-1})^{\frac1{n+2}}\frac{1}{|\Omega|^{\frac1n}},\ \ \varepsilon=\left(\sqrt{n-1}K(\Theta)\frac1{\sqrt{\mu_1\mu_0^{-1}}}\right)^{-2}\frac{|B|}{n^24^{n+1}}r_i^{n+2}.
$$

\end{Theorem}

\begin{Remark}
Combining Theorem \ref{thm:5.4}, \ref{thm:5.7} and \eqref{eqn:5.20}-\eqref{eqn:5.21}, we can obtain some similarly results with Theorem \ref{thm:5.8}.
\end{Remark}

\begin{proof}

We observe that, by using ${\rm div}(\frac{Du}{\sqrt{1-|Du|}})=n$,
\begin{equation}\label{eqn:5.21}
\begin{aligned}
|D(w-v)|^2=&|D(\frac{Du}{\sqrt{1-|Du|}})-D(x-z)|^2
=&|D(\frac{Du}{\sqrt{1-|Du|}})|^2-n\\
=&\frac{n-1}nH^2-2H_2=|\bar h|^2
\end{aligned}
\end{equation}

Combining Theorem \ref{thm:5.3}, \ref{thm:5.7} and \eqref{eqn:5.20}-\eqref{eqn:5.21} yield to
Theorem \ref{thm:5.8}. We have used $M<1$.
\end{proof}

From the Theorem \ref{thm:5.8}, we can see that stability estimate of compact hypersurface $\partial\Omega$ held by the assumption
$\|H-H_0\|_{L^1}<\varepsilon $ for $n\geq4$. In the last of this paper, inspired by Magnanini-Poggesi's method in \cite{MP2024} (or \cite{JLXZ}), we give another  stability estimate which doesn't need  this assumption. More precisely, we obtain the following result.
\begin{Theorem}\label{thm:5.13}
Let $\Omega\subset\mathbb R^n$, $n\geq2$, be a bounded domain with boundary of class $C^{2,\alpha}$.
Let $u$ be a solution of \eqref{eqn:5.1} and \eqref{eqn:5.2}. Let 
$$
z=\frac{1}{|\Omega|}\left(\int_{\Omega}x dx-\int_{\Omega}\frac{Du}{\sqrt{1-|Du|^2}} dx\right).
$$
If $z\in\Omega$, then we have\\
\begin{equation*}
\begin{aligned}
\rho_e-\rho_i
\leq&C(n,d_{\Omega},r_i,\Omega,K)\left\{
\begin{aligned}
&\|H-H_0\|^{\frac1{n+1}}_{L^1(\partial\Omega)},\ \ n\geq 3,\\
&\|H-H_0\|_{L^1(\partial\Omega)}^{\frac12}\max\left\{\log\|H-H_0\|_{L^1(\partial\Omega)}^{-\frac12},1\right\},\ \ n=2,
\end{aligned}\right.
\end{aligned}
\end{equation*}
where $H$ is the mean curvature function and $H_0$ is reference mean curvature.
\end{Theorem}
\begin{Remark}
Combining Theorem \ref{thm:5.4} and \eqref{eqn:5.26}-\eqref{eqn:5.29}, we can obtain some similarly results with Theorem \ref{thm:5.13}.
\end{Remark}

\begin{proof}
Let 
$$
z=\frac{1}{|\Omega|}\left(\int_{\Omega}x dx-\int_{\Omega}\frac{Du}{\sqrt{1-|Du|^2}} dx\right).
$$
and $z\in\Omega$. 

A simple computations give
\begin{equation}\label{eqn:5.26}
\int_{\Omega}D(w-v)dx=\int_{\Omega}(x-z)dx-\int_{\Omega}\frac{Du}{\sqrt{1-|Du|^2}}dx=\overrightarrow{0}.
\end{equation}

By triangle inequality, we can see that $d_{\Omega}\leq 2\rho_e$.  It follows from \eqref{eqn:05.11} and Lemma \ref{rem:5.7} that
\begin{equation}\label{eqn:5.27}
\begin{aligned}
\rho_e-\rho_i\leq&\frac{2(1+d_{\Omega})}{\rho_e}\left(\max_{x\in\bar\Omega}h-\min_{x\in\bar\Omega}h\right)\\
\leq&\frac{4(1+d_{\Omega})}{d_{\Omega}}\left(\max_{x\in\bar\Omega}h-\min_{x\in\bar\Omega}h\right)\\
\leq&C(n,d_{\Omega},r_i)\left\{
\begin{aligned}
&\|Dh\|^{\frac2{n+1}}_{L^2}\|Dh\|^{\frac{n-1}{n+1}}_{L^{\infty}},\ \ n\geq 3,\\
&\|Dh\|_{L^2}\log\left(e|\Omega|^{\frac12}\frac{\|Dh\|_{L^{\infty}}}{\|Df\|_{L^2}}\right),\ \ n=2.
\end{aligned}\right.
\end{aligned}
\end{equation}

On the other hand, we have following point-wise estimate of $|Dh|$: for any $y\in\bar\Omega$
\begin{equation}\label{eqn:5.28}
\begin{aligned}
|Dh(y)|=&\left|\frac{x-z}{\sqrt{1+|x-z|}}-Du(y)\right|\leq\left|\frac{x-z}{\sqrt{1+|x-z|}}\right|+|Du(y)|\\
\leq&\frac1{|\Omega|}\int_{\Omega}|y-x|dx+\frac1{|\Omega|}\int_{\Omega}\frac{|Du|}{\sqrt{1-|Du|^2}}dx\\
\leq&d_{\Omega}+K.
\end{aligned}
\end{equation}

By using Lemma \ref{rem:5.6} and \eqref{eqn:CC}, from \eqref{eqn:5.26}-\eqref{eqn:5.28}, we obtain that
\begin{equation}\label{eqn:5.29}
\begin{aligned}
\rho_e-\rho_i\leq&C(n,d_{\Omega},r_i)\left\{
\begin{aligned}
&\|w-v\|^{\frac2{n+1}}_{L^2}(d_{\Omega}+K)^{\frac{n-1}{n+1}},\ \ n\geq 3,\\
&\|w-v\|_{L^2}\max\left\{\log\left(\frac{d_{\Omega}+K}{\|w-v\|_{L^2}}\right),1\right\},\ \ n=2,
\end{aligned}\right.\\
\leq&C(n,d_{\Omega},r_i,\Omega,K)\left\{
\begin{aligned}
&\|D(w-v)\|^{\frac2{n+1}}_{L^2},\ \ n\geq 3,\\
&\|D(w-v)\|_{L^2}\max\left\{\log\left(\frac{1}{\|w-v\|_{L^2}}\right),1\right\},\ \ n=2.
\end{aligned}\right.
\end{aligned}
\end{equation}
This combining with Theorem \ref{thm:5.3} yield that
\begin{equation*}
\begin{aligned}
\rho_e-\rho_i
\leq&C(n,d_{\Omega},r_i,\Omega,K)\left\{
\begin{aligned}
&\|H-H_0\|^{\frac1{n+1}}_{L^1(\partial\Omega)},\ \ n\geq 3,\\
&\|H-H_0\|_{L^1(\partial\Omega)}^{\frac12}\max\left\{\log\|H-H_0\|_{L^1(\partial\Omega)}^{-\frac12},1\right\},\ \ n=2.
\end{aligned}\right.
\end{aligned}
\end{equation*}

This complete the proof of the assertion.
\end{proof}

\vskip 2mm
\noindent{\bf Acknowledgement.} The last author would like to thank  Professor Shanze Gao for their helpful conversations on this work.

\vskip2mm
\noindent{\bf Competing Interests}
The authors declare that there are no conflicts of interest.

\end{document}